\documentclass{amsart}

\newfont{\cyr}{wncyr10}

\newcommand{\map}[1]{\;\xrightarrow{#1}\;} 

\newcommand{\iso}{\cong}
\newcommand{\piso}{\sim}
\newcommand{\Gal}{\mathrm{Gal}}
\newcommand{\Hom}{\mathrm{Hom}}
\newcommand{\mil}{\lim\limits_\leftarrow} 
\newcommand{\dlim}{\lim\limits_\rightarrow}  

\newcommand{\Q}{\mathbf Q}
\newcommand{\R}{\mathbf R}
\newcommand{\C}{\mathbf C}
\newcommand{\Z}{\mathbf Z}
\newcommand{\co}{\mathcal O}
\newcommand{\D}{\mathcal D}
\newcommand{\A}{\mathbf{A}}

\newcommand{\GL}{\mathrm{GL}}

\newcommand{\Frob}{\mathrm{Fr}}
\newcommand{\tors}{\mathrm{tors}}
\newcommand{\inert}{\mathcal{I}}
\newcommand{\tam}{{c^\mathrm{tam}}}

\newcommand{\cm}{\mathcal{M}}
\newcommand{\cl}{\mathcal{L}}

\newcommand{\gm}{\mathfrak{m}}

\newcommand{\End}{\mathrm{End}}
\newcommand{\Aut}{\mathrm{Aut}}
\newcommand{\Pic}{\mathrm{Pic}}
\newcommand{\length}{\mathrm{length}}

\newcommand{\ph}{\mathfrak{H}}
\newcommand{\Hecke}{\mathbb{T}}
\newcommand{\N}{\mathbf{N}}
\newcommand{\Norm}{\mathrm{Norm}}

\newcommand{\gp}{\mathfrak{P}}
\newcommand{\gq}{\mathfrak{Q}}
\newcommand{\p}{\mathfrak{p}}
\newcommand{\q}{\mathfrak{q}}
\newcommand{\ord}{\mathrm{ord}}

\newcommand{\sel}{\mathcal{F}}
\newcommand{\f}{\mathrm{f}}

\newcommand{\tr}{\mathrm{tr}}
\newcommand{\loc}{\mathrm{loc}}
\newcommand{\unr}{\mathrm{unr}}
\newcommand{\can}{\mathrm{can}}
\newcommand{\Tw}{\mathrm{Tw}}

\newcommand{\bT}{\mathbf{T}}
\newcommand{\bW}{\mathbf{W}}
\newcommand{\Ind}{\mathrm{Ind}}

\theoremstyle{plain}
\newtheorem{Thm}{Theorem}[subsection]
\newtheorem{Prop}[Thm]{Proposition}
\newtheorem{Lem}[Thm]{Lemma}
\newtheorem{Cor}[Thm]{Corollary}

\newtheorem{BigThm}{Theorem}

\theoremstyle{definition}
\newtheorem{Def}[Thm]{Definition}

\theoremstyle{remark}
\newtheorem{Rem}[Thm]{Remark}

\input xy
\xyoption{all}

\title[Iwasawa theory of Heegner points]
{Iwasawa theory of Heegner points on abelian varieties of 
$\mathrm{GL}_2$-type}
\author{Benjamin Howard}
\address{Department of Mathematics\\ Harvard University\\ Cambridge, MA\\
02138}
\subjclass[2000]{11G05, 11G10, 11R23}

\begin{document}

\begin{abstract}
In an earlier paper the author proved one divisibility of 
Perrin-Riou's Iwasawa main conjecture for Heegner points on 
elliptic curves.  In the present paper, that result is 
generalized to abelian varieties of 
$\mathrm{GL_2}$-type (i.e. abelian varieties with real multiplication
defined over totally real fields) under the hypothesis that the abelian
variety is associated to a Hilbert modular form via a construction of
Zhang.
\end{abstract}

\thanks{This research was partially conducted by the author for the Clay
Mathematics Institute.}
\maketitle

\setcounter{tocdepth}{1}
\setcounter{section}{-1}
\tableofcontents

\section{Introduction}

Let $E$ be a CM field with $[E:\Q]=2g$, $F\subset E$ the maximal real 
subfield, and $\epsilon$ the quadratic character associated to $E/F$.  
Let $N$ be an integral ideal of $F$ which is prime to the 
relative discriminant of $E/F$, and which satisfies the 
\emph{weak Heegner hypothesis} that $\epsilon(N)=(-1)^{g-1}$.  Given
a Hilbert modular eigenform $\phi$ of parallel weight 2 for $\Gamma_0(N)$, 
the recent work of Zhang associates to $\phi$ an isogeny class of
abelian varieties over $F$ occuring as quotients of the Jacobian of
a certain Shimura curve $X$ associated to the data $(N,E)$.
These abelian varieties have good reduction away from $N$ and admit
real multiplication by the totally real field 
$F_\phi$ generated by the Hecke eigenvalues of $\phi$.
Fix one such quotient $\mathrm{Jac}(X)\map{}A$, let $\co\subset F_\phi$
be an order with $\co\hookrightarrow\End_F(A)$, and choose an
$\co$-linear polarization of $A$.

We abbreviate $G_E=\Gal(\bar{E}/E)$. For any rational prime 
$p$, the $p$-adic Tate module of $A$ decomposes as a direct sum
of $G_E$-submodules
$$T_p(A)\iso\bigoplus_{\gp\mid p}T_\gp(A)$$
where the sum is over the primes of $F_\phi$ above $p$.
Fix a prime $\gp$ of $F_\phi$, let $\co_\gp$ be the completion 
of $\co$ at $\gp$, and let $p$ be the rational prime below $\gp$.
We assume
\begin{enumerate}
\item\label{some primes} 
the order $\co_\gp$ is the maximal order of $F_{\phi,\gp}$ and that
$p$ does not divide  $2$, the class number of $E$, 
the index $[\co_E^\times:\co_F^\times]$, the absolute norm of $N$,
or the degree of the fixed polarization of $A$,

\item\label{galois image}
 the image of $\rho_\gp:G_E\map{}\Aut_{\co_\gp}(T_\gp(A))
\iso GL_2(\co_\gp)$ is equal to the subgroup $G_\gp\subset GL_2(\co_\gp)$ 
consisting of matrices whose determinant lies in 
$\Z_p^\times\subset\co_\gp^\times$.
\end{enumerate}
We remark that $G_\gp$ is the largest image one could hope for,
as the determinant of $\rho_\gp$ is equal to the cyclotomic character
$G_E\map{}\Z_p^\times$.  Furthermore the results of \cite{ribet}
suggest that when $A$ has exactly real multiplication, i.e.
$F_\phi\iso\End_{\bar{E}}(A)\otimes\Q_p$, then condition 
(\ref{galois image})  should hold for all but finitely many $\gp$.
Note that this condition implies that $G_E$ acts transitively on
the nonzero elements of $A[\gp]$, 
and hence $A(L)[\gp]=0$ for any abelian extension $L/E$.

For every finite extension $L/E$ we have
the two $\gp$-power Selmer groups which fit into the descent sequences
$$0\map{}A(L)\otimes_\co\co_\gp\map{}S_\gp(A_{/L})\map{}
\mil\mbox{\cyr Sh}(A_{/L})[\gp^k]\map{}0$$
$$0\map{}A(L)\otimes_\co(\Phi_\gp/\co_\gp)\map{}
\mathrm{Sel}_{\gp^\infty}(A_{/L})
\map{}\mbox{\cyr Sh}(A_{/L})[\gp^\infty]\map{}0$$
in which $\Phi_\gp$ is the field of fractions of $\co_\gp$.
The abelian variety $A$ comes equipped with a family of Heegner points 
defined over ring class fields of $E$.  Let $h[1]$ be the Heegner point
of conductor $1$, defined over the Hilbert class field $E[1]$.
Generalizing the work of Kolyvagin, Kolyvagin and Logachev, and Zhang
we will prove the following theorem in Section 2:

\begin{BigThm}\label{bigtheorem A}
Assume $\Norm_{E[1]/E}(h[1])\in A(E)$ has infinite order.
Then $S_\gp(A_{/E})$ is free of rank one over $\co_\gp$,
$\mbox{\cyr Sh}(A_{/E})[\gp^\infty]$ is finite, and there is an isomorphism
$$\mathrm{Sel}_{\gp^\infty}(A_{/E})\iso(\Phi_\gp/\co_\gp)\oplus M\oplus M$$
in which the order of $M$ is bounded by the index of the $\co_\gp$-submodule
of $S_\gp(A_{/E})$ generated by $\Norm_{E[1]/E}(h[1])$.
\end{BigThm}

Now let $\p$ be a prime of $F$ above $p$ and assume, in addition to
conditions (\ref{some primes}) and (\ref{galois image}) above, 
that $p$ is unramified in $E$.
Denote by $E[\p^k]$ the ring class field of conductor $\p^k$.
Then $\cup E[\p^k]$ contains a unique subfield $E_\infty/E$
with $\Gamma=\Gal(E_\infty/E)\iso\Z_p^f$, where $f$ is the residue degree of 
$\p$.  Let $\Lambda=\co_\gp[[\Gamma]]$ be the $f$-variable Iwasawa algebra,
and let $E_k\subset E_\infty$ be the fixed field of $\Gamma^{p^k}$.
Since we assume that $p$ does not divide the class number of $E$,
$E_k$ is the maximal $p$-power subextension of $E[p^{k+1}]/E$,
and we define $h_k$ to be the norm from $E[\p^{k+1}]$ to 
$E_k$ of the Heegner point of conductor $\p^{k+1}$.  Let $H_k$
be the $\Lambda$-module generated by all $h_j$ with $j\le k$, 
and set $H_\infty=\mil H_k$.  
Define finitely-generated $\Lambda$-modules 
$$S_{\gp,\infty}=\mil S_\gp(A_{/E_k})
\hspace{1cm}
X=\mil \Hom_{\co_\gp}(\mathrm{Sel}_{\gp^\infty}
(A_{/E_k}), \Phi_\gp/\co_\gp).$$  Let $X_{\Lambda-\tors}$
denote the $\Lambda$-torsion submodule of $X$.  In Section 3 we 
generalize the results of Bertolini, Nekov\'{a}\v{r}, and the author.  
The main result is

\renewcommand{\labelenumi}{(\alph{enumi})} 

\begin{BigThm}\label{bigtheorem B}
Suppose that $\p$ is the unique prime of $F$ above $p$,
and that $A$ has ordinary reduction at $\p$.  Assume further
that $h_k\in A(E_k)$ has infinite order for some $k$. 
Then 
\begin{enumerate}
\item $H_\infty$ and $S_{\gp,\infty}$ are torsion-free, rank one
$\Lambda$-modules,
\item $X$ has rank one as a $\Lambda$-module,
\item $X_{\Lambda-\tors}$ decomposes as
$$X_{\Lambda-\tors}\piso M\oplus M\oplus M_\gp$$
in which $M$ has $\mathrm{char}(M)$ prime to $\gp\Lambda$
and $\mathrm{char}(M_\gp)$ is a power of $\gp\Lambda$,
\item $\mathrm{char}(M)$ is fixed by the involution of
$\Lambda$ induced by inversion in $\Gamma$,
\item $\mathrm{char}(M)$ divides the characteristic ideal of 
$S_{\gp,\infty}/H_\infty$,
\end{enumerate}
where $\piso$ denotes pseudo-isomorphism of $\Lambda$-modules
and $\mathrm{char}$ denotes characteristic ideal. 
\end{BigThm}

A few remarks are in order concerning Theorem \ref{bigtheorem B}.
Following the conjectures of Perrin-Riou in \cite{pr87}, we
conjecture that equality holds in part (e), up to powers of
$\gp\Lambda$.  The recent success of Cornut and Vatsal in proving 
Mazur's conjecture on the nonvanishing of Heegner points gives us
hope that the hypothesis of some $h_k$ having infinite order
is always satisfied.  The hypothesis that 
some $h_k$ has infinite order is not needed for the proofs
of parts (c) and (d).  We expect that the assumption that 
$F$ has a unique prime above $p$ is not needed.

Even in the case where $F=\Q$ and $\phi$ has
rational coefficients (i.e. the case of an elliptic curve over $\Q$),
the above results are still stronger than those of \cite{howard}.
The reason is that we have replaced the classical Heegner hypothesis
that all primes dividing the level $N$ are split in $E$ by
the weaker hypothesis that $\epsilon(N)=1$.  Results
similar to those of Theorem \ref{bigtheorem B} in the
case where $\epsilon(N)=-1$ have recently been obtained by
Bertolini and Darmon in \cite{bert-dar-iwasawa}.

The methods used in the proofs of the  two main theorems 
draw very heavily from methods of Mazur and Rubin in \cite{mazur-rubin}.  
Furthermore, large portions require only trivial trivial modifications from 
arguments of \cite{howard}, and when this is the case we will only give
sketches of the proofs.

\bigskip

The following notation will remain in effect thoughout: 
$F$ is a totally real number field of degree $g$ 
and discriminant $d_F$, $\co_F$ is the ring of 
integers of $F$, $\A$ is the adele ring of $F$, $\A_\f$ the subring 
of finite adeles, and $\A_\infty$ the infinite component. 
If $v$ is any place of $F$ we denote by $F_v$ the completion of $F$ at $v$,
and if $A$ is any $F$-algebra we let $A_v=A\otimes_F F_v$.
If $M$ is an abelian group set $\hat{M}=M\otimes_\Z \hat{\Z}$.  In 
particular, $\A_\f\iso\hat{F}$.  

If $L$ is a perfect field we let 
$\bar{L}$ be an algebraic closure and
$G_L=\Gal(\bar{L}/L)$. If $L$ is a number field and $I$ is an ideal 
of the ring of integers of $L$, then we denote by $\N(I)$ 
the absolute norm of $I$.  Given a topological $G_L$-module, $M$, and any
place $v$ of $L$, we let $\loc_v:H^i(L,M)\map{}H^i(L_v,M)$
be the localization map.

We denote by $\ph$ and $\ph^\pm$ the upper half-plane and the 
union of the upper and lower half-planes, respectively.
If $E/F$ is a quadratic extension with $E$ totally complex, 
then we say that $E$ is a \emph{CM-extension} of $F$.
For any rational prime $p$, the $p$-adic Tate
module of $\mu_{p^\infty}$ is denoted $\Z_p(1)$.  If $M$ is any 
$\Z_p$-module we set $M(1)=M\otimes\Z_p(1)$. 

\section{Hilbert modular forms and Heegner points}
\label{Modular Part}

In Section \ref{Modular Part} we summarize some of the work of Shimura
and Zhang, closely following \cite{zhang2} and \cite{zhang1}, to which
we refer the reader for proofs. Useful references on
Hlbert modular forms and abelian varieties with real multiplication
include \cite{goren} and \cite{van-der-geer}.
Useful references on Shimura curves include 
\cite{shimura-quaternion} and \cite{shimura}, 
especially Chapter 9.  The standard reference on quaternion algebras
is \cite{vigneras}.


\subsection{Hilbert modular forms}


For any integral ideal $N\subset\co_F$ let
$$K_0(N)=\left\{ \left(\begin{array}{cc}a&b\\c&d\end{array}\right)
\in \GL_2(\hat{\co}_F) \mid c\equiv 0 \pmod{N}\right\}.$$
Identify $\A^\times$ with the center $Z(\A)\subset\GL_2(\A)$, and
for any $\theta=(\theta_v)\in \A_\infty$ set
$$r(\theta)=\left(\begin{array}{cc}
\cos \theta & \sin \theta \\ -\sin \theta & \cos \theta\end{array}\right)
\in \mathrm{SO}_2(\A_\infty).$$

\begin{Def}
By a \emph{Hilbert modular form} of (parallel) weight $k$ and level $N$ 
we mean a smooth function $\phi$ on $\GL_2(\A)$ satisfying
\begin{enumerate}
\item  $\phi$ is left invariant by $\GL_2(F)$ and right invariant
by $K_0(N)Z(\A)$,
\item for $r(\theta)\in \mathrm{SO}_2(\A_\infty)$,
$$\phi(g\cdot r(\theta))=\phi(g)\cdot\prod_{v|\infty}e^{ik\theta_v},$$
\item $\phi$ is of \emph{moderate growth} in the sense that for every
$c>0$ and every compact $\Omega\subset\GL_2(\A)$, there is a constant $M$
such that $$\phi\left(\left(\begin{array}{cc}a&0\\0&1\end{array}\right)g\right)
=O(|a|^M)$$ for all $g\in\Omega$ and $a\in\A^\times$ with $|a|>c$.
\item for every $h\in \GL_2(\A_\f)$ the function
$$x+iy\mapsto |y|^{-k/2}\phi\left(\left(\begin{array}{cc}y&x\\0&1\end{array}
\right)h\right)$$ is holomorphic in $x+iy\in\ph^g$.
\end{enumerate}
\end{Def}

To any Hilbert modular form $\phi$ there is an associated 
complex-valued function $a_\phi$, defined on the integral ideals of $F$.
The value $a_\phi(m)$ is called the $m^\mathrm{th}$ 
\emph{Fourier coefficient} of $\phi$, and these coefficients determine 
$\phi$ uniquely.  There is a notion of cusp form \cite[\S 3.1.1]{zhang1},
and the space of Hilbert modular cusp forms of 
weight $k$ and level $N$ is denoted $S_k(K_0(N))$.

Fix a level $N$, and let $m$ be an integral ideal of $\co_F$.
Let $\hat{\co}_F$ be the closure of $\co_F$ in $\A_\f$. 
Define a subset of $M_2(\hat{\co}_F) $   by
$$H(m)=\left\{ \left(\begin{array}{cc}a&b\\c&d\end{array}\right): 
(d,N)=\hat{\co}_F,\ c\in N\hat{\co}_F,\ (ad-bc)\hat{\co}_F= 
m\hat{\co}_F\right\},$$  and define the Hecke operator
$T_m$ acting on $S_k(K_0(N))$ by
$$(T_m\phi)(g) =\mathrm{N}(m)^{k/2-1}\int_{H(m)}\phi(gh)\ dh$$
where $dh$ is Haar measure on $\GL_2(\A_\f)$ normalized so that 
$K_0(N)$ has measure $1$.
Let $\Hecke_k(K_0(N))$ denote the $\Q$-subalgebra of 
$\End_\C(S_k(K_0(N)))$ generated by the $T_m$ with $m$ prime to $N$.
The Fourier coefficients of $T_m\phi$ are given by 
$$a_{T_m\phi}(n)=\sum_{a|(m,n)}\mathrm{N}(a)^{k-1}a_\phi(mn/a^2),$$
and the Hecke operators satisfy the formal identity
$$\sum\frac{T_m}{m^s}=\prod_{\gp|N}(1-T_\gp \N(\gp)^{-s})^{-1}
\prod_{\gp | \hspace{-4pt}\not\hspace{4pt} N}
(1-T_\gp \N(\gp)^{-s}+\N(\gp)^{1-2s})^{-1}.$$

If $N_1$ is a proper divisor of $N$, $\phi$ is a cusp form of level $N_1$,
and $d\in\GL_2(\A_\f)$ is such that $d^{-1}K_0(N)d\subset K_0(N_1)$,
then the function $\phi(gd)$ is a cusp form (of the same weight)
of level $N$.  The subspace of $S_k(K_0(N))$ generated by such functions
as $N_1$ and $d$ vary is called the space of \emph{old forms}.
The orthogonal complement of this subspace is denoted 
$S_k^\mathrm{new}(K_0(N))$.
We say that $\phi\in S_k^\mathrm{new}(K_0(N))$ is a \emph{newform} if
$a_\phi(1)=1$ and if $\phi$ is a simultaneous eigenform for all operators in
$\Hecke(K_0(N))$.  If this is the case then the Fourier coefficients of $\phi$
are algebraic integers, and generate an order in a totally real number field.
Furthermore, if $\sigma\in \Gal(\bar{\Q}/\Q)$ is any automorphism, then the
$a_\phi(m)^\sigma$ are the Fourier coefficients of another newform which
we denote by $\phi^\sigma$.  By the strong multiplicity one theorem,
if $\phi$ is a newform of level $N$ then $\phi$, a priori
only an eigenform for $T_m$ with $(m,N)=1$, is in fact an eigenform for
all $T_m$.  Also $\phi$ is an eigenvector of the involution $w_N$ defined
by $$(w_N\phi)(g)=\phi\left(g\left(\begin{array}{cc}0&1\\t&0
\end{array}\right)\right)$$ where $t\in\A$ is such that $t_\f$, the 
projection of $t$ to $\A_\f$, generates $N\hat{\co}_F$ and
has component $-1$ at the archimedean places.  Let $\gamma\in\{\pm 1\}$
be such that $w_N\phi=\gamma \phi.$
For $\phi$ of weight $2$, the $L$-function of $\phi$ is defined by 
\begin{eqnarray*}L(s,\phi)
&=&\prod_{\ell|N}\frac{1}{1-a_\phi(\ell)\N(\ell)^{-s}}
\prod_{\ell|\hspace{-4pt}\not\hspace{4pt} N}\frac{1}
{1-a_\phi(\ell)\N(\ell)^{-s}+\N(\ell)^{1-2s}}\\
&=&\sum_{m}\frac{a_\phi(m)}{\N(m)^s}.
\end{eqnarray*}
Let $d_N$ denote the absolute norm of $N$. The completed $L$-function
$$L^*(s,\phi)= d_N^{s/2}d_F^s
\left(\frac{\Gamma(s)}{(2\pi)^s}\right)^g L(s,\phi)$$
has analytic continuation and satisfies the functional equation 
$$L^*(s,\phi)=\gamma L^*(2-s,\phi).$$


\subsection{Heegner points on Shimura curves}
\label{shimura section}

Let $F$ be a totally real number field of degree $g$,
$\xi:F\hookrightarrow\R$ a fixed embedding, and $N$ an ideal of $\co_F$.
Fix a CM-extension $E/F$ whose relative
discriminant, $D_{E/F}$, is prime to $N$, and let 
$\epsilon:F^\times\backslash\A^\times
\map{}\{\pm 1\}$ be the quadratic character associated to $E/F$.  We assume 
the \emph{weak Heegner hypothesis} that $\epsilon(N)=(-1)^{g-1}$.

Let $N_B$ be the squarefree product of primes $\gp| N$ which are inert
in $E$ and have $\mathrm{ord}_\gp(N)$ odd, and fix an integral ideal $N_E$
of $E$ with relative norm $N/N_B$.  Since $\epsilon(N_B)=
\epsilon(N)=(-1)^{g-1}$, there is a unique quaternion algebra $B/F$
which is ramified exactly at the prime divisors of $N_B$ and the archimedean
primes other than $\xi$.  Fix an isomorphism
$$B\otimes_\Q \R\iso M_2(F_\xi)\oplus \mathbf{H}^{g-1}$$ where 
$\mathbf{H}$ denotes the
real quaternions.  The group of units $B^\times$ can be given the structure
of the set of rational points of a reductive algebraic group $G$ over 
$F$, $G(F)\iso B^\times,$ and the projection 
$$G(\A_\infty)\iso(B\otimes_\Q \R)^\times 
\map{}\GL_2(F_\xi)$$ defines an action of 
$G(\A_\infty)$ (and so also of $B^\times$) 
on $\ph^\pm$. We let $U_\infty$ be the 
stabilizer of $i$ and identify $$\ph^\pm \iso G(\A_\infty)/U_\infty.$$ 
The projection  $G(\A_\infty)\map{}\ph^\pm$ admits a smooth section $s$ 
defined by \begin{equation}\label{section} 
s(x+iy)=\left(\left(\begin{array}{cc}
y&x\\0&1\end{array}\right),1,\ldots,1\right).\end{equation}

At every place at which $B$ is ramified
$E\otimes_F F_v$ is a field, and so there exists an embedding $q:E\map{}B$. 
There is a unique point $w(q)\in\ph$ which is fixed by $q(\alpha)$ for
every $\alpha\in E^\times$.  The embedding $q$ and its conjugate embedding
share the same fixed point, and exactly one of them is \emph{normalized}
in the sense that $$q(\alpha)\left[\begin{array}{c}w(q)\\1\end{array}
\right]=\alpha \left[\begin{array}{c}w(q)\\1\end{array}\right]$$
where $q(\alpha)$ is viewed as an element of $GL_2(\R)$ on the left hand side,
and $\alpha$ is a scalar multiplier on the right hand side.  We assume
that $q$ is the normalized choice.

Let $\co_B$ be a maximal order of $B$ containing 
$q(\co_E)$ and define an order $R$ of reduced discriminant $N$ by 
$$R=q(\co_E)+q(N_E)\co_B.$$  
Let $U\subset G(\A_\f)$ be image of $\hat{R}^\times$ under the isomorphism
$\hat{B}^\times\iso G(\A_\f)$, and
let $Z$ be the center of $G$, so that $Z(\A_\f)\iso\hat{F}^\times$.
Define the complex curve $X(\C)$ to be the quotient
\begin{eqnarray*}X(\C)&=& G(F)\backslash\ph^\pm\times G(\A_\f)/Z(\A_\f)U
\ \cup \{\mathrm{cusps}\}\\
&=& G(F)\backslash G(\A)/Z(\A)UU_\infty \ \cup \{\mathrm{cusps}\}.
\end{eqnarray*}
This is a compact and possibly disconnected Riemann surface.
The set of cusps is nonempty only when $F=\Q$ and $B=M_2(\Q)$.
If $F=\Q$ and every prime divisor of $N$ splits in $E$,
then $X(\C)$ is none other than the classical level $N$ modular curve 
$X_0(N)$. 

If $(z,g)\in\ph\times G(\A_\f)$, we write $[(z,g)]$ for the class of 
$(z,g)$ in $X(\C)$.
The normalizer of $U$ in $G(\A_\f)$ acts
on $X(\C)$ by $\alpha\cdot [(z,g)]=[(z,g\alpha^{-1})]$.  In Shimura's
language, this is the automorphism  $J(\alpha)=J_{UU}(\alpha)$.
Let $\pi_0(X(\C))$ denote the set of connected components of $X(\C)$.
The reduced norm $\nu:G(\A)\map{}\A^\times $ induces a bijection 
$$
\pi_0(X(\C))\iso F^\times \backslash \A^\times/ \nu(Z(\A)UU_\infty ).
$$
If $F_X$ is the abelian extension of $F$ with
$$\Gal(F_X/F)\iso F^\times \backslash \A^\times/ \nu( Z(\A)UU_\infty)$$
via the Artin symbol,
we let $\sigma:G(\A_\f)\map{}\Gal(F_X/F)$ be the map taking
$\alpha\mapsto \nu(\alpha)^{-1}$.
It is easily checked that $\nu(U)=\hat{\co}_F^\times$, 
and so $F_X$ is a subfield of the narrow Hilbert class
field of $F$.  We define an action of $\Gal(F_X/F)$ on $\pi_0(X(\C))$
by the commutativity of 
\begin{equation}\label{component reciprocity}
\xymatrix{
X(\C)\ar[r]^{J(\alpha)}\ar[d]^\nu&X(\C)\ar[d]^\nu\\
\pi_0(X(\C))\ar[r]^{\sigma(\alpha)}&\pi_0(X(\C)).}\end{equation}

Let $m$ be an integral ideal of $F$ which is prime to $N$. At every
prime $\ell | m$ the algebra $B$ is split, and the component of $U$ 
is a maximal compact open subgroup of $G(F_\ell)$.
Let $\Delta(m)$ (resp. $\Delta(1)$) be the set of elements 
of $\hat{\co}_B$ with component
$1$ away from $m$, and whose determinant generates $m$ (resp. is a unit)
at every prime divisor of $m$.
We define a correspondence $T_m$ on $X(\C)$ by
\begin{equation}\label{correspondence}
T_m\cdot[(z,g)] =\sum_{\Delta(m)/\Delta(1)}
[(z,g\gamma)]\end{equation} as a divisor on $X(\C)$.

We let $T\subset G$ be the torus defined by
$q(E^\times)=T(F)$, and let $w(q)$ denote the unique fixed point
of $T(F)$ in $\ph$.  Define 
$$\mathrm{CM}_E= T(F)\backslash G(\A_\f)/Z(\A_\f),$$ 
the set of CM-points by $E$. We map $\mathrm{CM}_E$ to $X(\C)$
via $g\mapsto (w(q),g)$ and call the image the CM-points of $X(\C)$. 
Define an action of  $\Gal(E^\mathrm{ab}/E)$ on the CM-points by  
\begin{equation}\label{reciprocity}
[(w(q),g)]^{[s,E]}= [(w(q),s\cdot g)]
\hspace{.5cm} \forall s\in T(\A_\f)
\end{equation} where
$[\ \  ,E]:T(F)\backslash T(\A_\f)\iso \Gal(E^{\mathrm{ab}}/E)$
is the Artin symbol. In particular note that $Z(\A_\f)\iso\A_\f^\times$ 
acts trivially on all CM points.

If $x$ is a CM-point represented
by $(w(q),g)\in \ph\times G(\A_\f)$, 
we define the endomorphism ring of $x$ to be the
preimage of $\hat{R}$ under the map 
$g^{-1}q g:E\map{}\hat{B}.$  It is an order of $E$ of the form
$\co_c=\co_F+c\co_E$ for some integral ideal $c\subset\co_F$ called the 
\emph{conductor} of $x$. 
Let $$T[c]=q(\hat{\co}_c^\times)\subset T(\A_\f).$$
The abelian extension of $E$ associated to $T[c]Z(\A_\f)$ by class field
theory is called the \emph{ring class field} of conductor $c$ and is denoted
$E[c]$.  It is Galois over $F$, and is the natural field of 
definition of $x$.

We want to give an explicit construction of some CM-points of various 
conductors. Let $\ell$ be a prime of $F$ not dividing 
$D_{E/F}N$, and fix an isomorphism
$B_\ell\iso M_2(F_\ell)$ in such a way that $R_{\ell}$ is identified
with $M_2(\co_{F,\ell})$, and so that 
$$q(\co_{E,\ell})=\left\{\left(\begin{array}{cc}x&0\\0&y\end{array}\right):
x,y\in \co_{F,\ell}\right\}$$
in the case where $\ell$ splits in $E$, or
$$q(\co_{E,\ell})=\left\{\left(\begin{array}{cc}x&yu\\y&x\end{array}\right):
x,y\in \co_{F,\ell}\right\}$$
for some $u\in\co_{F,\ell}^\times$ not a square, in the case where
$\ell$ is inert in $E$.
Fix a uniformizer $\varpi$ of $F_\ell$, and
let $h[\ell^k]$ be the element of $B_\ell$ such that
$$h[\ell^k]\mapsto\left\{ \begin{array}{ll}
\left(\begin{array}{cc} \varpi^k & 1\\ &1 \end{array}\right) &
\mathrm{\ if\ }\ell\mathrm{\ splits\ in\ } E\\
\left(\begin{array}{cc} \varpi^k & \\ &1 \end{array}\right) &
\mathrm{\ if\ }\ell\mathrm{\ inert\ in\ } E\end{array}
\right.$$
under the above isomorphism.  View $h[\ell^k]$ as an element of $G(\A_\f)$
with trivial components away from $\ell$, and extend $h$ multiplicatively 
to a map on all integral ideals prime to $D_{E/F}N$.

Direct calculation yields the following properties of the points $h[m]$:
\begin{Prop}
There is a collection of CM-points of $h[m]\in X(\C)$, where $m$ runs
over all positive integers prime to $D_{E/F}N$,
such that $h[m]$ has conductor $m$ and such that
\begin{eqnarray*}\lefteqn{[\co_m^\times:\co_{m\ell}^\times]\cdot
\mathrm{Norm}_{E[m\ell]/E[m]}(h[m\ell]) }
\hspace{1cm}\\
&= &\left\{\begin{array}{ll}
T_\ell (h[m]) &\mathrm{if\ } \ell\not|\ m\mathrm{\ and\ is\ inert\ in\ }E \\
T_\ell(h[m])- h[m]^{\sigma_\ell}- h[m]^{\sigma_\ell^*}
&\mathrm{if\ }\ell\not|\ m\mathrm{\ and\ is\ split\ in\ }E\\
T_\ell (h[m])-h[m/\ell]&\mathrm{if\ }\ell\mid m
\end{array}\right.\end{eqnarray*}
as divisors on $X(\C)$, where $\sigma_\ell$ and $\sigma_\ell^*$
are the Frobenius automorphisms of the primes of $E$ above $\ell$.
\end{Prop}

The existence of a canonical model for the complex curve $X(\C)$ is due
to Shimura:

\begin{Thm}
There is a smooth projective variety $X$, defined and connected over $F$,
whose complex points are isomorphic to $X(\C)$ as a Riemann 
surface.  The action of $\Gal(\bar{F}/F)$ on the geometric components
factors through $\Gal(F_X/F)$ and agrees with  the action 
determined by (\ref{component reciprocity}).
If $x\in X(\C)$ is a CM-point then $x$ is defined over 
$E^\mathrm{ab}$ and the action of $\Gal(E^\mathrm{ab}/E)$ agrees with the
action (\ref{reciprocity}).
\end{Thm}
\begin{proof} This is Theorem 9.6 of \cite{shimura}. \end{proof}


\subsection{Abelian varieties associated to newforms}
\label{abelian construction}


Fix an integral ideal $N\subset\co_F$, and assume
that either $[F:\Q]$ is odd or that $\ord_v(N)$ is 
odd for some finite prime $v$ of $F$. 
Then we may fix a CM extension $E/F$ of relative discriminant $D_{E/F}$ which 
satisfies  the weak Heegner hypothesis $\epsilon(N)=(-1)^{g-1}$,
where $\epsilon$ is the quadratic character associated to $E/F$.
Abbreviate $\Hecke=\Hecke_2(K_0(N))$.
If $\phi$ is a Hilbert modular newform of weight $2$ and level $N$ on
$\A$, we  let $\alpha_\phi:\Hecke\map{}\C$
be the character giving the action of $\Hecke$ on $\C\cdot\phi$.  
We denote by $F_\phi$ the totally real field generated by 
$\alpha_\phi(T_m)$ with $(m,N)=1$.

Let $X$ be the canonical model over $F$ of the
complex curve $X(\C)$ described in Section \ref{shimura section}.
The curve $X$ splits into its geometric components over 
the field $F_X$ defined in the previous section, i.e.
$X\times_F F_X=\amalg X_i$ with each $X_i$ geometrically irreducible,
and any extension $L/F$ for which $X(L)\not=\emptyset$ must contain
$F_X$. Define $J_X$ to be the abelian variety over $F$ obtained by the 
restriction of scalars of $\mathrm{Jac}(X_0)$ (for some fixed 
component $X_0$) from $F_X$ to $F$.
Then $J_X$ has good reduction away from $N$, 
and for any algebraic extension $L/F$ with 
$X(L)\not=\emptyset$, $$J_X(L)=\prod_i\mathrm{Jac}(X_i)(L)
=\prod_i\Pic^0(X_i\times_{F_X} L).$$
We denote by $\Hecke_X$ the $\Q$-algebra generated by the Hecke
correspondences (\ref{correspondence}) acting on $J_X$.

By the Jacquet-Langlands correspondence, for every algebra homomorphism
$\alpha:\Hecke_X\map{}\C$ there exists a weight $2$ level $N$ newform
$\phi$ such that (slightly abusing notation) 
$\alpha(T_m)=\alpha_\phi(T_m)$.  This associates to every maximal ideal
of $\Hecke_X$ a unique Galois conjugacy class of newforms and also
gives a surjective algebra map 
$\Hecke\map{}\Hecke_X$, thus endowing the Lie algebra of $J_X$
 with an action of $\Hecke$.

\begin{Thm}(Zhang)
There is an isogeny $J_X\map{\sim}\oplus_{\phi} A_{\phi}$
such that the induced map on Lie algebras
is $\Hecke$-equivariant, where the sum is over 
all Galois conjugacy classes of
newforms of weight $2$ and level dividing $N$.
If $\phi$ is new of level $N$, the Lie algebra of $A_{\phi}$ is free of
rank one over $F_\phi\otimes_\Q \C$.
Furthermore, for each $\phi$
there is an equality of $L$-functions
$$L_N(s,A_\phi)=\prod_{\sigma:F_\phi\hookrightarrow\C}L_N(s,\phi^\sigma)$$
where the subscript $N$ indicates that the Euler factors at primes 
dividing $N$ have been removed.
\end{Thm}

Fix a newform $\phi$.
In order to obtain Heegner points on $A_\phi$ it
suffices to exhibit embedding $X\map{}J_X$ which is defined over $F$ and is 
compatible with the action of the Hecke operators.  
Since the curve $X$ typically has no cusps,, 
there is no natural choice of $F$-rational point on $X$ to provide 
such an embedding. Instead, one uses the \emph{Hodge class} $\xi\in\Pic(X)$:
the unique (up to constant multiple) class whose degree is constant
on each geometric component and which satisfies
$$T_m\xi=\mathrm{deg}(T_m)\xi$$ for every Hecke correspondence with
$m$ prime to $D_{E/F}N$.  In the absence of cusps and elliptic fixed points, 
the Hodge class is simply the canonical divisor on each geometric component.
Write $X(\C)=\cup_i X_i(\C)$ as a disjoint union of connected components,
and let $\xi_i$ be the restriction of $\xi\in \Pic(X(\C))$
to the $i^\mathrm{th}$ component.  Denote by $d$ the degree of $\xi_i$. 
There is a  unique morphism $X\map{}J_X$, defined over $F$, 
which on complex points takes $p_i\in X_i(\C)$
to the divisor $d p_i-\xi_i\in J_X(\C)$.

Applying the composition $X\map{}J_X\map{}A_\phi$ 
to the Heegner points described in the previous section yields the following.

\begin{Prop}\label{heegner points}
There is a family of points $h[m]\in A_\phi(E^\mathrm{ab})$, 
where $m$ runs over all positive integers prime to $D_{E/F}N$,
such that $h[m]$ is defined over $E[m]$, and
\begin{eqnarray*}\lefteqn{[\co_m^\times:\co_{m\ell}^\times]\cdot
\mathrm{Norm}_{E[m\ell]/E[m]}(h[m\ell]) }
\hspace{1cm}\\
&= &\left\{\begin{array}{ll}
a_\phi(\ell)h[m]&\mathrm{if\ } \ell\not|\ m\mathrm{\ and\ is\ inert\ in\ }E \\
a_\phi(\ell)h[m]- h[m]^{\sigma_\ell}- h[m]^{\sigma_\ell^*}
&\mathrm{if\ }\ell\not|\ m\mathrm{\ and\ is\ split\ in\ }E\\
a_\phi(\ell)h[m]-h[m/\ell]&\mathrm{if\ }\ell\mid m
\end{array}\right.\end{eqnarray*}
\end{Prop}

\begin{Thm}(Zhang) Let $L(s,\phi,E)=L(s,\phi)L(s,\epsilon,\phi)$. 
Then $L(s,\phi,E)$ has analytic continuation and a functional equation
equation in $s\mapsto 2-s$ with sign $\epsilon(N)(-1)^{g-1}=-1$.  
In particular $L(s,\phi,E)$ vanishes at $s=1$.
Assume that a prime $\p$ of $F$ is split in $E$ if either   
$\p$ divides  $2$ or $\ord_\p(N)> 1$. Then 
$$L'(1,\phi,E)\not=0\ \Longleftrightarrow\ 
\Norm_{E[1]/E}(h[1])\mathrm{\ has\ infinite\ order}.$$
\end{Thm}
\begin{proof} This is Theorem C of \cite{zhang1}. \end{proof}

One expects that the requirement 
that $\p$ splits when $\ord_\p(N)> 1$ is unnecessary.


\section{Bounding the Selmer group}\label{Euler part}


In Section 2 we prove Theorem \ref{bigtheorem A}.  The Heegner points
having been constructed, the remainder of the proof is 
essentially identical to arguments of \cite{howard}, and 
makes fundamental use of the observation of \cite{mazur-rubin} 
that Kolyvagin's derivative classes $\kappa_m$ satisfy certain 
``transverse'' local conditions at primes dividing $m$.

Throughout Section \ref{Euler part} we work with a fixed CM field $E$,
and denote by $F$ its maximal real subfield.  Fix a complex
conjugation $\tau\in G_F$ and a rational
prime $p$ which does not divide $2$, the class number of $E$, or the
index $[\co_E^\times:\co_F^\times]$.


\subsection{Kolyvagin systems}
\label{Kolyvagin systems}


In preparation for the Iwasawa theory of Section \ref{iwasawa},
we work in greater generality than is need for the proof of 
Theorem \ref{bigtheorem A}.  
By a coefficient ring we mean a complete, Noetherian, local
ring with finite residue field of characteristic $p$.  Let $R$ be such a 
ring, and suppose that $T$ is any topological $R$-module equipped
with a continuous $R$-linear action of $G_{E}$, unramified outside a finite
set of primes.

\begin{Def}
A \emph{Selmer structure} on $T$ is a pair $(\sel,\Sigma)$ where $\Sigma$
is a finite set of places of $E$ containing the archimedean places, 
the primes
at which $T$ is ramified, and all primes above $p$; and $\sel$ is 
a collection of \emph{local conditions} at the places of $\Sigma$.  That is,
for each $v\in\Sigma$ we have a choice of $R$-submodule
$$H^1_\sel(E_v,T)\subset H^1(E_v,T).$$ If $E^\Sigma$ 
denotes the maximal
extension of $E$ unramified outside of $\Sigma$, then we define the
\emph{Selmer module} $H^1_\sel(E,T)$ to be the kernel of the localization
$$H^1(E^\Sigma/E,T)\map{\oplus \loc_v}\bigoplus_{v\in\Sigma} 
H^1(E_v,T)/H^1_\sel(E_v,T).$$
\end{Def}

\begin{Rem}
Equivalently, one may define a Selmer structure to be a family
of local conditions $H^1_\sel(E_v,T)\subset H^1(E_v,T),$
one for \emph{every} place $v$, such that almost all local conditions
are equal to the unramified condition.  We usually take this point of view,
so that the set $\Sigma$ does not need to be specified.
\end{Rem}

\begin{Rem}\label{local archimedean}
Since $E$ is totally complex, $H^1(E_v,T)=0$
at every archimedean place $v$.
\end{Rem}

If $S$ is a submodule (resp. quotient) of $T$ then a Selmer structure
on $T$ induces a Selmer structure on $S$ by taking the preimages 
(resp. images) of the local conditions on $T$ under the natural 
maps on local cohomology. 
We refer to this as \emph{propagation} of Selmer structures.

The most important example of a local condition is the
\emph{unramified} condition:
let $v$ be a finite place of $E$, and denote by $E_v^\unr$ the maximal
unramified extension of $E_v$.  The unramified condition 
$H^1_\unr(E_v, T)$ is defined as the kernel of restriction
$$H^1(E_v,T)\map{}H^1(E_v^\unr,T).$$

For the remainder of this section we fix a Selmer structure 
$(\sel,\Sigma)$ on $T$, and assume that $T$ is finitely generated
over $R$.

\begin{Def}
A prime ideal $\ell$ of $\co_F$ is 
\emph{$k$-admissible} if it satisfies
\begin{enumerate}
\item $\ell$ does not divide $D_{E/F}$, is inert in $E$, and is not
in $\Sigma$,
\item $\N(\ell)+1\equiv 0 \pmod{p^k}$,
\item the Frobenius of the prime of $E$ above $\ell$ acts
trivially on $T/p^kT$.
\end{enumerate}A $1$-admissible prime will simply be 
called \emph{admissible}.
\end{Def}
We will routinely confuse an admissible prime of $F$ with the unique
prime of $E$ above it.  To avoid confusion, the Frobenius of the
unique prime of $E$ above an admissible $\ell$ will be denoted
$\Frob_E(\ell)$. The
set of $k$-admissible primes is denoted $\cl_k$,
and $\cm_k$ denotes the set of squarefree products of primes of
$\cl_k$.   If $\ell$ is admissible let
$G(\ell)$ be the $p$-Sylow subgroup of the cokernel of
$$(\co_F/\ell\co_F)^\times\map{}(\co_E/\ell\co_E)^\times.$$
This is a cyclic group of order equal to the maximal power of $p$ dividing
$\N(\ell)+1$, and is the same as the $p$-Sylow subgroup of 
$(\co_E/\ell\co_E)^\times$, since we assume that $p$ does not divide 
$\N(\ell)-1$.
For any $m\in\cm_1$, let $E(m)$ be the $p$-ring class field of conductor
$m$, i.e. the maximal $p$-power subextension of $E[m]/E$, 
and note that $E(1)=E$ since we assume that $p$
does not divide the class number of $E$.

\begin{Lem} For any $m\in\cm_1$, 
\begin{enumerate}\label{prelim lemma one}
\item there is a canonical isomorphism 
$$G(m)\stackrel{\mathrm{def}}{=}\Gal(E(m)/E)\iso \prod_{\ell|m}G(\ell)$$
\item if $\ell$ is a prime of $F$ not dividing $m D_{E/F}$ which is
inert in $E$ (in particular if $\ell$ is admissible and prime to $m$), 
then the unique prime of 
$E$ above $\ell$ splits completely in $E(m)$,
\item if $\ell$ is a prime divisor of $m$ and $\lambda$ is a prime of
$E(m)$ above $\ell$, then $E(m)_\lambda$ is a totally tamely ramified
abelian $p$-extension of $E_\ell$, and is a maximal such
extension,
\item $\Gal(E(m)/F)$ is a generalized dihedral group:
for any $\sigma\in\Gal(E(m)/E)$ and any complex conjugation 
$\tau\in\Gal(E(m)/F)$, one has $\tau \sigma\tau=\sigma^{-1}$.
\end{enumerate}
\end{Lem}
\begin{proof}  Elementary class field theory.
\end{proof}

\begin{Def}
Let $\ell\in\cl_1$, and let $\lambda$ be the unique prime of
$E(\ell)$ above $\ell$.  We define the \emph{transverse} local
condition at $\ell$, $H^1_\tr(E_\ell,T)$, to be the
kernel of restriction
$$H^1(E_\ell,T)\map{}H^1(E(\ell)_\lambda,T).$$
If $\sel$ is any Selmer structure on $T$ and $m\in\cm_1$, we define
a new Selmer structure $\sel(m)$ on $T$ by
$$H^1_{\sel(m)}(E_\ell,T)=\left\{\begin{array}{ll}
H^1_\tr(E_\ell,T)&\mathrm{if\ }\ell\mid m\\
H^1_\sel(E_\ell,T)&\mathrm{else.}\end{array}\right.$$
\end{Def}

In practice, we only define the transverse condition when
$T$ is annihilated by $|G(\ell)|$ and $\Frob_E(\ell)-1$.
Accordingly, for any admissible $\ell$, we let $I_\ell= p^k R$
where $k$ is the largest integer for which $\ell$ is $k$-admissible.
If $m\in\cm_1$ set $$I_m=\sum_{\ell\mid m}I_\ell\hspace{1cm}
\Delta_m=\bigotimes_{\ell\mid m}G(\ell)$$
so that $T/I_m T$ is annihilated both by $|G(\ell)|$ and by
$\Frob_E(\ell)-1$ for any $\ell$ dividing  $m$.  
By Lemma 1.2.4 of 
\cite{mazur-rubin}, if $\ell\in\cl_1$ and $I\subset R$ is any
ideal containing $I_\ell$, there is a decomposition
$$H^1(E_\ell,T/IT)\iso H^1_\unr(E_\ell,T/IT)\oplus H^1_\tr(E_\ell,T/IT).$$
Furthermore, Lemma 1.2.1 of \cite{mazur-rubin} gives canonical isomorphisms
$$H^1_\unr(E_\ell,T/IT)\iso T/IT\hspace{1cm}
H^1_\tr(E_\ell,T/IT)\otimes G(\ell)\iso T/IT,$$
both of which are given by evaluation of cocycles:
the first is evaluation at the Frobenius automorphism,
and the second sends $c\otimes\sigma_\ell\mapsto c(\sigma_\ell)$
where  $\sigma_\ell$ is a generator of $\Delta_\ell$.
If $\ell\in\cl_1$ and $I$ contains $I_\ell$, 
we define the \emph{edge map} (or finite-singular
comparision map) at $\ell$ to be the isomorphism
$$e_{\ell}:H^1_\unr(E_\ell,T/IT)\iso T/IT\iso 
H^1_\tr(E_\ell,T/IT)\otimes G(\ell).$$
For every  $m\ell\in\cm_1$,  consider the maps
\begin{equation}\label{ks relations}\xymatrix{
 &H^1_{\sel(m)}(E,T/I_m T)\otimes\Delta_m \ar[d]^{\loc_\ell}\\
 &H^1_\unr(E_\ell,T/I_{m\ell}T)\otimes\Delta_m\ar[d]^{e_{m,\ell}\otimes 1} \\
H^1_{\sel(m\ell)}(E,T/I_{m\ell}T)\otimes\Delta_{m\ell} \ar[r]^{\loc_\ell}&
H^1_\tr(E_\ell,T/I_{m\ell}T)\otimes\Delta_{m\ell}.}\end{equation}

\begin{Def}
Let $\cl\subset\cl_1$, and denote by $\cm$ the set of squarefree
products of primes in $\cl$.  
We define a \emph{Kolyvagin system} $\kappa$ for $(T,\sel,\cl)$
to be a collection of cohomology classes $$\kappa_m\in
H^1_{\sel(m)}(E,T/I_mT)\otimes\Delta_m$$ one for each $m\in\cm$,
such that for any $m\ell\in\cm$ the images of $\kappa_m$ and
$\kappa_{m\ell}$ in $H^1_\tr(E_\ell,T/I_{m\ell}T)\otimes\Delta_{m\ell}$
under the maps of (\ref{ks relations}) agree.  We denote the
$R$-module of all Kolyvagin systems 
for $(T,\sel,\cl)$ by $\mathbf{KS}(T,\sel,\cl)$.
\end{Def}


\subsection{The main bound}
\label{main bound}


Let $S$ be the ring of integers of a finite extension 
$\Phi/\Q_p$, and let $T$ be a free $S$-module of rank $2$ equipped
with a continuous $S$-linear action of $G_E$.  Let $\Sigma$ be a 
finite set of primes of $E$ containing the infinite places, the
primes above $p$, and all primes at which $T$ is ramified.
Let $\gm$ be the maximal ideal of $S$, and fix a uniformizer $\pi\in\gm$.
We set $\D=\Phi/S$, $V=T\otimes_S\Phi$, $W=V/T$.

Fix a Selmer structure $(\sel,\Sigma)$ on $V$, and propagate this
to Selmer structures, still denoted $\sel$, on $T$ and $W$.
The fact that the Selmer structure on $T$ is propagated from $V$ implies
that the local conditions $H^1_\sel(E_v,T)$ are \emph{cartesian}
on the category of quotients of $T$ (see Definition 1.1.4 and Lemma 3.7.1
of \cite{mazur-rubin}).  Consequently, the isomorphism 
$T/\gm^k\iso W[\gm^k]$ identifies the Selmer structure on $T/\gm^k$
propagated from $T$ with the Selmer structure on $W[\gm^k]$ propagated
from $W$.

We assume throughout this section that
the module $T$ satisfies the following hypotheses:
\begin{description}
\item[H1] there is an extension $L/E$ which is Galois over $F$,
such that $G_L$ acts trivially on $T$ and 
$H^1(L(\mu_{p^\infty})/E,T/\gm T)=0$,
\item[H2] $T/\gm T$ is an absolutely irreducible representation of
$(S/\gm)[[G_E]]$, and the action of $G_E$ extends to an action of $G_F$.
Furthermore, the action of $\tau$ splits $T/\gm T$ into two one-dimensional
eigenspaces,
\item[H3] there is a perfect, symmetric, $S$-bilinear pairing
$$(\ ,\ ):T\times T\map{}S(1)$$ such that $(a^\sigma,b^{\tau\sigma\tau})
=(a,b)^\sigma$ for all $a,b\in T$ and $\sigma\in G_E$.  
The induced pairing $T/\gm T\times T/\gm T\map{}(S/\gm)(1)$
on the residual representation satisfies $(a^\tau,b^\tau)=(a,b)^\tau$,
\end{description}

The pairing of {\bf H3} can be thought of as a $G_E$-equivariant
pairing $T\times\Tw(T)\map{}S(1)$ where $\Tw(T)$ is the Galois module
whose underlying $S$-module is $T$, but on which $G_E$ acts through the
automorphism $\sigma\mapsto\tau\sigma\tau$.  This automophism,
together with the map $T\map{}\Tw(T)$ which is the identity on
underlying $S$-modules, 
induces a ``change of group'' $(G_E,T)\map{}(G_E,\Tw(T))$, and hence
an isomorphism $$H^i(E,T)\iso H^i(E,\Tw(T)).$$  At every prime $v$ of 
$E$, there is a similar isomorphism $H^i(E_{v^\tau},T)\iso H^i(E_v,\Tw(T))$,
and similar remarks hold with $T$ replaced by $W$ or $V$.
Tate local duality therefore gives perfect pairings
\begin{eqnarray}\label{local duality}
H^1(E_v,T)\times H^1(E_{v^\tau},W)&\map{}&\D\\
H^1(E_v,V)\times H^1(E_{v^\tau},V)&\map{}&\Phi\nonumber
\end{eqnarray} at every place $v$.
We assume that the Selmer structure $\sel$ on $T$ satisfies
\begin{description}
\item[H4] at every place $v$, the local conditions 
$H^1_\sel(E_v,V)$ and $H^1_\sel(E_{v^\tau},V)$
are exact orthogonal complements under the pairing (\ref{local duality}),
\item[H5] at every place $v$ of $F$, the module 
$\oplus_{w\mid v}H^1_\sel(E_w,T/\gm T)$ is stable under the
action of $\Gal(E/F)$. 
\end{description}

\begin{Prop}\label{structure}
There is an integer $r$ and a finite $S$-module $M$ such that
$$H^1_\sel(E,W)\iso \D^r\oplus M\oplus M.$$
\end{Prop}
\begin{proof} Define a Selmer structure $\sel$ on $\Tw(W)$ by identifying
$$H^1(E_{v^\tau},W)\iso H^1(E_v,\Tw(W))$$  everywhere locally.
By the main result of \cite{flach}, there is a generalized Cassels pairing
$$H^1_\sel(E,W)\times H^1_\sel(E,\Tw(W))\map{}\D$$
whose kernels on the left and right are exactly the submodules of
$S$-divisible elements. The global change of group isomorphism identifies
$H^1_\sel(E,W)$ with $H^1_\sel(E,\Tw(W))$, and under this identification
the pairing above yields a pairing
$$H^1_\sel(E,W)\times H^1_\sel(E,W)\map{}\D.$$
A straightforward (if tedious) modification of the methods of \cite{flach}
shows that the resulting pairing is alternating; a similar calculation 
is done in Theorem 1.4.3 of \cite{howard}.
\end{proof}

\begin{Thm}\label{my thesis}
Suppose we have a set of primes $\cl\subset\cl_1$ with $\cl_e\subset\cl$
for $e\gg 0$.  Let $\cm$ denote the set of squarefree products of primes
in $\cl$.  Suppose that
there is a collection of cohomology classes
$$\{\kappa_m\in H^1(E,T/I_m T)\otimes \Delta_m \mid m\in\cm\}$$
such that $\kappa_1\not=0$ and there exists an integer $d\ge 0$, 
independent of $m$, such that the family $p^d\kappa_m$ is a Kolyvagin 
system for $(T,\sel,\cl)$.
Then $\kappa_1\in H^1_\sel(E,T)$ and $H^1_\sel(E,T)$
is free of rank one over $S$.  Furthermore, there is 
an isomorphism $$H^1_\sel(E,W)\iso \D\oplus M\oplus M$$
with $\mathrm{length}_S(M)\le \mathrm{length}_S\big(
H^1_\sel(E,T)/ S\cdot\kappa_1\big)$.
\end{Thm}
\begin{proof} The case $F=\Q$ and $d=0$ is 
Theorem 1.6.1 of \cite{howard}.  The only nontrivial modifications 
needed are to deal with $d>0$, and these are essentially contained in 
the proof of Corollary 4.6.5 of \cite{rubin}.

Fix some integer $e$ large enough
that $\kappa_1$ has nontrivial image in $H^1(E,T/IT)$, where $I=p^eS$, and
such that $\tilde{\cl}\stackrel{\mathrm{def}}{=}\cl_{e+d}$ is contained
in $\cl$. Let $\tilde{\kappa}_m$ denote the image of $\kappa_m$
in $H^1(E,T/IT)\otimes\Delta_m$.  The claim is that the 
family $\tilde{\kappa}_m$ is a Kolyvagin system for $(T/IT,\sel,\tilde{\cl})$
(over the ring $S/I$).  
Let $\tilde{\cm}$ denote the set of squarefree 
products of primes in $\tilde{\cl}$.
We must show that if $m\in\tilde{\cm}$ then $\tilde{\kappa}_m$ lies in 
$H^1_{\sel(m)}(E,T/IT)\otimes\Delta_m$.
If $m=1$ this follows from the fact that $H^1(E,T)/H^1_\sel(E,T)$
is torsion free (as the Selmer structure $\sel$ was assumed to 
be propagated from a Selmer structure on $V$) and the hypothesis
that $p^d\kappa_1\in H^1_\sel(E,T)$.
If $m\not=1$ then $I_m$ 
is generated by (say) $p^k$ with $e+d\le k$.  Set $I'=p^{e+d}S\supset I_m$.
Multiplication by $p^d$ on $T$ induces a map 
\begin{equation}\label{cheapie}
H^1(E,T/IT)\otimes\Delta_m\map{}H^1(E,T/I' T)\otimes\Delta_m
\end{equation}
taking $\tilde{\kappa}_m$ to the image of $p^d\kappa_m$ modulo $I'$,
and this image lies in $H^1_{\sel(m)}(E,T/I' T)\otimes\Delta_m$ 
by hypothesis.
By the cartesian property of $\sel(m)$ (see the remarks at the beginning
of this subsection and Lemma 3.7.4 of \cite{mazur-rubin}) it 
follows that $$\tilde{\kappa}_m\in H^1_{\sel(m)}(E,T/IT)\otimes\Delta_m.$$
It is easily seen that the maps on local cohomology analogous to 
(\ref{cheapie}) are compatible with the edge maps of Section
\ref{Kolyvagin systems} and so the classes $\tilde{\kappa}_m$
form a Kolyvagin system.

The remainder of the proof is now \cite{howard} Theorem 1.6.1
almost verbatim.
For every $m\in\tilde{\cm}$ one has a (noncanonical) decomposition 
$$H^1_{\sel(m)}(E,W)[I]\iso 
H^1_{\sel(m)}(E,T/IT)\iso (S/I)^\epsilon\oplus M_m\oplus M_m$$
where $\epsilon\in\{0,1\}$ is independent of $m$, and the
first isomorphism is given by \cite{mazur-rubin} Lemma 3.5.3.  
For $m\in\tilde{\cm}$, we define the \emph{stub Selmer module} at $m$ to be
$$\mathrm{Stub}(m)=\gm^{\mathrm{length}_S(M_m)}\cdot H^1_{\sel(m)}(E,T/IT)
\otimes\Delta_m.$$  By further shrinking $\tilde{\cl}$ we may assume
that $\tilde{\cl}\subset\cl_{2e+d}$, and the key point is that for every
$m\in\tilde{\cm}$ the class $\tilde{\kappa}_m$ belongs to the stub Selmer
module at $m$.  In particular, $\mathrm{Stub}(1)$ is nonzero and
$M_1$ has length strictly less than that of $S/I$.
This implies that $\epsilon=1$ and that the module $M$ of the statement 
of the theorem is finite.  Furthermore,
the image of $\kappa_1$ in 
$H_\sel^1(E,T/IT)$ actually lies in $\gm^{\mathrm{length}_S(M_1)}
H_\sel^1(E,T/IT)$.  Taking limits as
$e\to\infty$ shows that $\kappa_1\in \gm^{\mathrm{length}_S(M)}
H_\sel^1(E,T)$.
\end{proof}


\subsection{Application to Heegner points}
\label{heegner}


Now let $N$ be an integral ideal of $F$, and let
$A_{/F}$ be an abelian variety associated to a Hilbert modular form, 
$\phi$, of level $N$.  Fix an embedding 
$\co\hookrightarrow\End_F(A)$ for some order $\co$ of $F_\phi$, 
and an $\co$-linear polarization of $A$.
Fix a prime $\gp$ of the ring of integers of $F_\phi$, let $p$ be the 
rational prime below $\gp$, and assume that conditions
(\ref{some primes}) and (\ref{galois image}) of the introduction hold.
Denote by $\co_\gp$ the completion of $\co$ at $\gp$, let $\Phi_\gp$
be the field of fractions of $\co_\gp$, and set $\D_\gp=\Phi_\gp/\co_\gp$.
We let $\Sigma$ be any finite set of places of $E$ containing the
archimedean places, the primes above $p$, and the divisors of $N\co_E$.
Let $T=T_\gp(A)$, $V=T\otimes_{\co_\gp}\Phi_\gp$, and 
$W=V/T\iso A[\gp^\infty]$.
For any ideal $m$ of $\co_F$ we abbreviate $a_m=a_\phi(m)\in\co_\gp$.

Our choice of polarization of $A$ gives a perfect, skew-symmetric, 
$G_F$-equivariant pairing
\begin{equation}\label{primitive Weil}
T_\gp(A)\times T_\gp(A)\map{}\Z_p(1)
\end{equation}
under which the action of $\co_\gp$ is self-adjoint.

\begin{Lem}\label{prelim lemma two}
Let $\mathrm{Tr}:\co_\gp(1)\map{}\Z_p(1)$ be the map induced by the 
trace from $\co_\gp$ to $\Z_p$.
\begin{enumerate}
\item The module $T_\gp(A)$ is free of rank $2$ over $\co_\gp$,
\item  there is a perfect, skew-symmetric, $\co_\gp$-bilinear, 
$G_F$-equivariant pairing
$$e_\gp:T_\gp(A)\times T_\gp(A)\map{}\co_\gp(1)$$ such that the pairing 
(\ref{primitive Weil}) factors as $\mathrm{Tr}\circ e_\gp$,
\item the action of any complex conjugation in $G_F$
splits $T_\gp(A)$ into two rank-one eigenspaces,
\item for any admissible $\ell$, the Frobenius of $\ell$ (over $F$)
acts as a conjugate of complex conjugation on $T_\gp(A)/I_\ell T_\gp(A)$,
and $a_\ell\in I_\ell$.
\end{enumerate}
\end{Lem}
\begin{proof} Fix a complex parametrization $\C^d/L\iso A(\C)$.  
Then $T_\gp(A)\iso L\otimes_\co \co_\gp$ is free of rank $2$ over $\co_\gp$.
If $\mathfrak{d}\in\Phi_\gp$ is a 
generator for the (absolute) inverse different of
$\co_\gp$, then the map
$$\Hom_{\co_\gp}( T_\gp(A),\co_\gp)\map{}
\Hom_{\Z_p}(T_\gp(A),\Z_p)$$ defined by $f\mapsto \mathrm{Tr}\circ 
(\mathfrak{d}\cdot f)$
is an isomorphism.  For any $s\in T_\gp(A)$, let $f_s$ denote the image of
$s$ under $T_\gp(A)\map{}\Hom_{\Z_p}(T_\gp(A),\Z_p(1))$, and let
$g_s$ be the unique lift of $f_s$ to $\Hom_{\co_\gp}(T_\gp(A),\co_\gp(1))$.
The pairing $e_\gp(s,t)=g_s(t)$ now has the desired properties of (b).
Part (c) follows from the Galois equivariance of this pairing.  The
claims of  (d) follow from the fact that the Frobenius of $\ell$ over $E$ acts
trivially on $T_\gp(A)/I_\ell T_\gp(A)$, and the Frobenius relative
to $F$ acts on $T_\gp(A)$ with characteristic polynomial
$1-a_\ell X +\N(\ell)X^2$.
\end{proof}

\begin{Def}
We define the \emph{canonical} Selmer structure $(\sel^\can,\Sigma)$
on $V$ by taking the unramified local condition at any place $v$ 
of $E$  not dividing $p$,
and taking  the image of the local Kummer map
$$A(E_v)\otimes_{\co}F_\phi\map{}H^1(E_v,V)$$ if $v$ does divide $p$.
We also denote by $\sel^\can$ the Selmer structures on 
$T$ and $W$ obtained by propagation.
\end{Def}

\begin{Prop}\label{true selmer}
At every place $v$ of $E$, the sequence 
$$0\map{}H^1_{\sel^\can}(E_v,W)\map{}
H^1(E_v,W)\map{}H^1(E_v,A)[\gp^\infty]\map{}0$$ 
is exact. Consequently, there is an exact sequence 
$$0\map{}A(E)\otimes_\co (\Phi_\gp/\co_\gp)\map{}
H^1_{\sel^\can}(E,W)\map{}\mbox{\cyr Sh}(A_{/E})[\gp^\infty]\map{}0.$$
\end{Prop}
\begin{proof} This is Proposition 1.6.8 of \cite{rubin}. \end{proof}

\begin{Lem}\label{canonical condition}
For any prime $v$ of $E$ not dividing $p$, 
$$H^1_{\sel^{\can}}(E_v,V)=H^1_{\sel^{\can}}(E_v,W)=  0.$$
\end{Lem}
\begin{proof} This follows from Corollary 1.3.3 of \cite{rubin}. \end{proof}

For $\ell\in\cl_1$ set $u_\ell=(\N(\ell)+1)/|G(\ell)|\in\Z_p^\times$.  
If $m\in\cm_1$, set $u_m=\prod_{\ell\mid m}u_\ell$ and define
$$h(m)=u_m^{-1}[\co_E^\times:\co_m^\times]\cdot \Norm_{E[m]/E(m)} h[m]
\in A(E(m))\otimes_\co\co_\gp,$$
where  $h[m]$ is the Heegner point of Section \ref{abelian construction},
and let $c_m$ be the image of $h(m)$ under the Kummer map 
$A(E(m))\otimes_\co\co_\gp\map{}H^1(E(m),T)$.
The collection $$\{ c_m\in H^1(E(m),T)\mid m\in \cm_1\}$$ is
the \emph{Heegner point Euler system}, and satisfies the relation
$$\Norm_{E(m\ell)/E(m)} c_{m\ell}=u_\ell^{-1}a_\ell\cdot c_m.$$
For each admissible $\ell$ we fix a generator $\sigma_\ell\in G(\ell)$
and define Kolyvagin's \emph{derivative operator} $D_\ell\in \co_\gp[G(\ell)]$
by $$D_\ell=\sum_{i=1}^{|G(\ell)|-1}i\sigma_\ell^i.$$
The derivative operator satisfies the telescoping identity
$(\sigma_\ell-1)D_\ell=|G(\ell)|-N_\ell$, where $N_\ell\in\co_\gp[G(\ell)]$
is the norm element.  If $m\in \cm_1$
let $D_m\in \co_\gp[G(m)]$ be defined by $D_m=\prod_{\ell|m}D_\ell$.
Exactly as in \cite{howard}, the class
$$D_m c_m\in H^1(E(m),T/I_m T)$$ is fixed by the action
of $G(m)$, and there is a unique class
$\kappa'_m \in H^1(E,T/I_m T)$ mapping to $D_m c_m$ under restriction.
In order to remove the dependence on the choices of $\sigma_\ell$,
set $$\kappa_m=\kappa'_m\otimes_{\ell\mid m}\sigma_\ell\in 
H^1(E,T/I_mT)\otimes\Delta_m.$$
The collection  $\{\kappa_m\mid m\in\cm_1\}$ 
(or at least some multiple of it) is the 
\emph{Heegner point Kolyvagin system}.

\begin{Lem}
Let $\tam$ be the product of the local Tamagawa factors of $A_{/E}$.
For every $m\in\cm_1$, $\tam\cdot \kappa'_m\in H^1_{\sel(m)}(E,T/I_mT)$.
\end{Lem}
\begin{proof}  We identify $T/I_mT\iso W[I_m]$.
Suppose $v$ is a prime of $E$ not dividing $mp$.  We must show
that $\loc_v(\tam\cdot\kappa'_m)\in H^1_{\sel^\can}(E_v,W[I_m])$.  If $v$ is 
archimedean, then by Remark \ref{local archimedean} there is 
nothing to prove, and so we assume $v$ is nonarchimedean.  
Let $w$ be a prime of $E(m)$ above $v$.  
It follows from Lemma \ref{canonical condition} and 
Proposition  \ref{true selmer} (which hold 
with $E$ replaced by $E(m)$) that the image of $h(m)$ under the composition
$$A(E(m))\map{}A(E(m)_w)\map{}H^1(E(m)_w,W[I_m])\map{}
H^1(E(m)_w,W)$$ is trivial for every $w$ above $v$, 
and so the image of $D_m c_m$ under 
$$H^1(E(m),W[I_m])\map{}H^1(E(m)_w,W)$$
is trivial.  
Since $v$ is unramified in $E(m)$, this shows that
$\kappa'_m$ lies in $H^1_\unr(E_v,W)$ and hence (since the order of
$H^1_\unr(E_v,W)$ is the $p$-part of the local Tamagawa factor at $v$
by Proposition I.3.8 of \cite{milne} and the Herbrand quotient)  
$\tam\kappa_m'$
has trivial image in $H^1(E_v,W)$. By Lemma
\ref{canonical condition} and the definition of propagation of Selmer
structures,  this shows that the localization of
$\tam \kappa'_m$ at $v$ lies in $H^1_\sel(E_v,W[I_m])$.

Suppose that $v=\ell$ is a divisor of $m$, and let $\lambda$ be the unique
prime of $E(\ell)$ above $\ell$.  It suffices to show that
$\kappa'_m$ has trivial image in $H^1(E(\ell)_\lambda,W[I_m])$, and
since $\lambda$ splits completely in $E(m)$ it suffices to 
check that $D_m c_m$ is trivial in the semilocalization
$$H^1(E(m)_\ell,W[I_m])\stackrel{\mathrm{def}}{=}
\bigoplus_{w|\ell}H^1(E(m)_{w},W[I_m]).$$
Since the image of the Kummer map in $H^1(E(m)_{w},W[I_m])$ is
unramified for every choice of $w$, it follows that
$\loc_{\ell}(c_m)\in H^1_\unr(E(m)_\ell,W[I_m])$. 
Evaluation at Frobenius gives an isomorphism of $G(m)$-modules
$$H^1_\unr(E(m)_\ell,W[I_m])\iso \bigoplus_{w|\ell} W[I_m],$$
where $G(m)$ acts on the right hand side by permuting the summands.
In particular $G(\ell)\subset G(m)$ acts trivially, since every prime of
$E(m/\ell)$ above $\ell$ is totally ramified in $E(m)$.  The action
of $D_\ell$ on $H^1_\unr(E(m)_\ell,W[I_m])$ is therefore
multiplication by $\frac{|G(\ell)|\cdot (|G(\ell)|-1)}{2}\in I_\ell$.
This shows that $D_m c_m=D_\ell D_{m/\ell} c_m$ is trivial in 
$H^1(E(m)_\ell,W[I_m])$.

Suppose $v$ divides $p$. By Proposition \ref{true selmer} it suffices to
show that the image of $\kappa'_m$ under the composition
$$H^1(E,W[I_m])\map{}H^1(E_v,W)\map{}
H^1(E_v,A)$$ is trivial. Consider the commutative diagram 
$$\xymatrix{
H^1(E_v,W[I_m])\ar[r]\ar[d]& {\bigoplus_{w|v} 
H^1(E(m)_w,W[I_m])}\ar[d]\\
H^1(E_v,A)\ar[r]& {\bigoplus_{w|v}H^1(E(m)_w,A)}.}$$
The image of $\loc_v(\kappa_m)$ under the top horizontal arrow is
$\oplus\loc_w(D_m c_m)$, and the image of this under the right vertical 
arrow is trivial, since $c_m$ is in the image of the global Kummer map.
Since $A$ has good reduction at $v$, Proposition I.3.8 of \cite{milne}
implies that the restriction map 
$$H^1(E_v, A)\map{}H^1(E_v^\unr,A)$$ is injective, and so the
bottom horizontal arrow of the diagram is also injective.  
This proves the claim. \end{proof}

\begin{Lem}\label{finite-singular}
For every $\ell\in\cm_1$ there is an 
$\co_\gp$-automorphism $\chi_\ell$ of  $T/I_\ell T$ such that the isomorphism
$$\phi: H^1_\unr(E_\ell,T/I_{m\ell} T)\iso T/I_{m\ell}T 
\map{\chi_\ell}T/I_{m\ell}T\iso H^1_\tr(E_\ell,T/I_{m\ell} T)\otimes
\Delta_\ell$$ satisfies
$\phi(\loc_\ell(\tam\cdot\kappa'_m))=
\loc_\ell(\tam\cdot\kappa'_{m\ell})\otimes \sigma_\ell$ 
for  every $m$ such that $m\ell\in\cm_1$.  Furthermore, if $m\in\cm_1$
then the maps $\chi_\ell:T/I_mT\map{}T/I_mT$ with $\ell|m$ pairwise commute. 
\end{Lem}
\begin{proof} This is exactly as in Proposition 4.4 of \cite{mccallum}.
\end{proof}

\begin{Thm}\label{little money}
Suppose $h(1)\not=0$.  Then
\begin{enumerate}
\item the $\co_\gp$-modules $A(E)\otimes_\co\co_\gp$ and 
$H^1_{\sel^\can}(E,T)$ are free of rank one,
\item the $\gp$-primary component of $\mbox{\cyr Sh}(A_{/E})$ is finite,
\item there is a finite $\co_\gp$-module $M$ such that
$$H^1_{\sel^\can}(E,W)\iso \D_\gp\oplus M\oplus M$$
and $\length_{\co_\gp}(M)\le\length_{\co_\gp}\left(
H^1_{\sel^\can}(E,T)/\co_\gp\cdot  h(1)\right)$.
\end{enumerate}
\end{Thm}
\begin{proof} For each $m\in\cm_1$ and
$\ell|m$, the automorphism $\chi_\ell$ of the preceeding 
Lemma induces an automorphism of $H^1(E,T/I_mT)$ which
we still denote by $\chi_\ell$.  Setting $\chi_m$ equal to the 
composition of $\chi_\ell$ as $\ell$ runs over all divisors of $m$,
the collection
$$\tam \chi_m^{-1}(\kappa_m)\in H^1_{\sel(m)}(E,T/I_mT)\otimes\Delta_m$$ 
is a Kolyvagin system for $(T,\sel,\cl_1)$.  Furthermore, $\kappa_1$ is equal
to the image of $h(1)$ under the Kummer map
$A(E)\otimes\co_\gp\map{}H^1_\sel(E,T).$

By Theorem \ref{my thesis} (applied to the family $\chi_m^{-1}(\kappa_m)$
with $p^d\Z_p=\tam\Z_p$) 
we need only check that the hypotheses 
{\bf H1--H5} are satisfied.  
Recall our assumption that the image of 
$G_E\map{}\Aut_{\co_\gp}(T_\gp(A))$ is equal to $G_\gp$, the subgroup of 
automorphisms whose determinant lies in $\Z_p^\times$.  Taking
$L=E(A[\gp^\infty])$, we have
$H^1(L/E,A[\gp])\iso H^1(G_\gp,A[\gp])$ which is trivial
(apply the inflation-restriction sequence to the subgroup
$\Z_p^\times\hookrightarrow G_\gp$ imbedded along the diagonal),
showing that {\bf H1} holds.  Hypothesis {\bf H2} follows from
Lemma \ref{prelim lemma two} and the fact that $G_\gp$ acts
transitively on the nontrivial elements of 
$A[\gp]$.  If $e_\gp$ denotes the pairing
of Lemma \ref{prelim lemma two}, then the pairing $(s,t)=e_\gp(s,t^\tau)$
satisfies the properties of {\bf H3}.  Hypothesis {\bf H4} is Tate
local duality, and {\bf H5} is trivially verified.
 \end{proof}


\section{Iwasawa theory}
\label{iwasawa}


Let $\phi$, $A$, $E/F$, $\co\subset F_\phi$, $p$, $\gp$, $\co_\gp$,
$\Phi_\gp$, and $\D_\gp$ be as in  Section \ref{heegner}.
In addition to conditions (\ref{some primes}) and (\ref{galois image})
of the Introduction, we assume that 
\begin{enumerate}
\item  there is a unique prime $\p=p\co$ of $F$ above $p$,
\item $A$ has good ordinary reduction at $\p$.
\end{enumerate}
It can be deduced from the results of Section 3.6.2 of \cite{goren} that the 
ordinary hypothesis implies that $a_\p\in\co_\gp^\times$.
Recall that $\Phi_\gp$ denotes the field of fractions of $\co_\gp$,
and $\D_\gp=\Phi_\gp/\co_\gp$.
Let $\Sigma$ be a finite  set of places of $E$ consisting of the archimedean
places and the divisors of $\p N\co_E$, and
define $E^\Sigma$ to be the maximal extension 
of $E$ unramified outside $\Sigma$.

Our main tool for studying the cohomology of the $\Lambda$-modules
$\bT$ and $\bW$ (defined in the next section) is to consider,
following \cite{mazur-rubin}, the cohomology of 
$\bT\otimes_\Lambda S$ as $S$ runs over discrete valuation
rings with $\Lambda$-algebra structures.  The added complication
of working over an Iwasawa algebra in several variables presents
several new technical hurdles, but apart from that point the arguments
follow \cite{howard} very closely.


\subsection{Heegner points in anti-cyclotomic extensions}
\label{iwasawa kolyvagin}


As before, for any ideal
$c$ of $\co_F$ we let $E[c]$ denote the ring class field of conductor
$c$, and $E(c)$ the maximal $p$-power subextension. 
The field $E[\p^\infty]=\cup E[\p^k]$ has 
$$\Gal(E[\p^\infty]/E[1])\iso\co_{E,\p}^\times/\co_{F,\p}^\times$$
and so $E[\p^\infty]$ contains a unique subfield, $E_\infty$,
with $\Gamma\stackrel{\mathrm{def}}{=}
\Gal(E_\infty/E)\iso \Z_p^g$.
We call this the anti-cyclotomic  extension of $E$.
Our running assumptions on $p$ imply that each prime of $E$ above $\p$
is totally ramified in $E_\infty$, that $E_k=E(\p^{k+1})$ is 
the fixed field of $\Gamma^{p^{k}}$, and that $E_k$ and 
$E(m)$ are linearly disjoint over
$E$ for any $m\in\cm_1$.  Define $E_k(m)=E_kE(m)=E(m\p^{k+1})$,
and set $E_\infty(m)=\cup E_k(m)$.  Exactly as in Lemma 
\ref{prelim lemma one}, one may easily check the following facts:
\begin{enumerate}
\item for any $m\in\cm_1$, $E_\infty(m)/F$ is of dihedral type,
\item if $m\ell\in\cm_1$ then the unique prime of $E$ above $\ell$
splits completely in $E_\infty(m)$.
\end{enumerate}
Let $\Lambda=\co_\gp[[\Gamma]]$ be the Iwasawa algebra, and
let $\iota:\Lambda\map{}\Lambda$ be the involution which is inversion 
on group-like elements.

As in Section 2.2 of \cite{howard}, we define $G_E$ and $\Lambda$-modules 
$$\bT=\mil\Ind_{E_k/E} T_\gp(A)\hspace{1cm}
\bW=\dlim\Ind_{E_k/E}A[\gp^\infty],$$
where $\Ind_{E_k/E}$ is the induction functor from $G_{E_k}$-modules to 
$G_E$-modules, and the limits are with respect to the natural corestriction
and restriction maps.  There is an isomorphism 
$\bT=T_\gp(A)\otimes_{\co_\gp}\Lambda$ with $G_E$ acting on the second factor
through $G_E\map{}\Lambda^\times\map{\iota}\Lambda^\times$, and
Shapiro's lemma gives canonical isomorphisms
\begin{eqnarray*}H^1(E(m),\bT)&\iso& \mil H^1(E_k(m),T_\gp(A))\\
H^1(E(m),\bW)&\iso& \dlim H^1(E_k(m),A[\gp^\infty]),
\end{eqnarray*}
and similar isomorphisms on semi-local cohomology.
Furthermore, as in Proposition 2.2.4 of \cite{howard},
there is a perfect, $G_E$-equivariant pairing
\begin{equation}\label{iwasawa pairing}
e_\Lambda:\bT\times\bW\map{}\D_\gp(1)\end{equation} satisfying
$e_\Lambda(\lambda t,a)=e_\Lambda(t,\lambda^\iota a)$ for all $t\in\bT$,
$a\in\bW$, and $\lambda\in\Lambda$.
The action of $G_E$ on  $\bT$ and $\bW$ factors through $\Gal(E^\Sigma/E)$.

For each integer $k\ge 0$
we define the Heegner point $h_k(m)\in A(E_k(m))\otimes_\co\co_\gp$ by 
$$h_k(m)=u_m^{-1}[\co_E^\times:\co_m^\times]\cdot 
\Norm_{E[m\p^{k+1}]/E_k(m)} h[m\p^{k+1}],$$
with $u_m\in\Z_p^\times $ as in Section \ref{heegner}.  
For $m\in\cm_1$ and $k\ge0$, let 
$$H_k(m)\subset  A(E_k(m))\otimes_\co\co_\gp$$ be the 
$\co_\gp[\Gal(E_k(m))/E(m)]$-module generated by $h_j(m)$ 
for $0\le j\le k$,
and set $H_\infty (m)=\mil H_k(m)$.  
Define $\Phi\in \co_\gp[\Gal(E(m)/E)]$
by the formula $$\Phi=\left\{\begin{array}{ll}
(N(\p)+1)^2-a_\p^2 & \mathrm{inert\ case}\\
(N(\p)-a_\p\sigma+\sigma^2)(N(\p)-a_\p\sigma^*+\sigma^{*2}) &
\mathrm{split\ case,}\end{array}\right.$$
in which split and inert refer to the behavior of $\p$ in $E$, and
$\sigma$ and $\sigma^*$ are the Frobenius elements in $\Gal(E(m)/E)$
of the primes above $\p$.
As $k$ varies, the points $h_k(m)$ 
are almost norm compatible, and can be modified to give elements of 
$H_\infty(m)$:

\begin{Prop}\label{heegner module}
There is a family
$\{c_m\in H_\infty(m)\mid m\in\cm_1\}$ satisfying
\begin{enumerate}
\item $c_m$ generates the torsion-free $\Lambda$-module 
$H_\infty(m)$,  and is nonzero if and only if 
$h_k(m)$ has infinite order for some $k$,
\item for $m\ell\in\cm_1$, $\Norm_{E_\infty(m\ell)/E_\infty(m)} 
c_{m\ell}=u_\ell^{-1}a_\ell\cdot c_m$,
\item the image of $c_m$ under $H_\infty(m)\map{}H_0(m)$
is $\Phi h(m)$, where $h(m)$ is the Heegner point of Section
\ref{heegner}.
\end{enumerate}
\end{Prop}
\begin{proof} This is proven exactly as in Section 2.3 of 
\cite{howard}. \end{proof}


\subsection{Ordinary Selmer modules}


For each prime $\q$ dividing $\p\co_E$, define $\co_\gp$-modules
$A[\gp^\infty]^\pm$ by taking 
$A[\gp^\infty]^+$ to be the kernel of reduction
$$A[\gp^\infty]\map{}\tilde{A}[\gp^\infty],$$
where $\tilde{A}$ is the reduction of $A$ at $\q$, and 
$A[\gp^\infty]^-=\tilde{A}[\gp^\infty]$.  Define 
$T_\gp(A)^\pm$ similarly, and note that although we supress
it from the notation, these modules depend on $\q$.  In fact,
they depend on fixing a place of $\bar{E}$ above $\q$, 
as does the localization map 
$H^i(E,A[\gp^\infty])\map{}H^i(E_{\q},A[\gp^\infty])$.  We will always
assume, without further comment, that consistent choices are made.
By definition, there are  exact sequences
$$
0\map{}T_\gp(A)^+\map{}T_\gp(A)\map{}T_\gp(A)^-\map{}0
$$
$$
0\map{}A[\gp^\infty]^+\map{}A[\gp^\infty]\map{}A[\gp^\infty]^-\map{}0.
$$
The pairing of Lemma \ref{prelim lemma two} induces perfect pairings
$$T_\gp(A)^\pm\times A[\gp^\infty]^\mp\map{}\D_\gp(1).$$
For each prime of $E$ dividing $\p$, the exact sequences above 
induce exact sequences
$$
0\map{}\bT^+\map{}\bT\map{}\bT^-\map{}0
$$
$$
0\map{}\bW^+\map{}\bW\map{}\bW^-\map{}0
$$
together with perfect pairings
$\bT^\pm\times\bW^\mp\map{}\D_\gp(1).$

\begin{Lem}\label{unramified}
For every place $v$ of $E$ not dividing $p$, the groups  
$$H^1(E_v,\bT)/H^1_\unr(E_v,\bT)\hspace{1cm}H^1_\unr(E_v,\bW)$$
have finite exponent.  If $v\not\in\Sigma$ these groups are trivial.
\end{Lem}
\begin{proof} All references in this proof are to \cite{rubin}.  Let $L$ be
an unramified finite extension of $E_v$.  
By Corollary 1.3.3 and local Tate duality,
$$H^1(L,T_\gp(A)\otimes\Phi_\gp)=H^1_\unr(L,T_\gp(A)\otimes\Phi_\gp)=0,$$
and so Lemma 1.3.5 (iii) implies that
$$H^1(L,T_\gp(A))/H^1_\unr(L,T_\gp(A))\iso \mathcal{W}^{\Frob=1}$$
where $\mathcal{W}$ is $A(E_v^\unr)[\gp^\infty]$ modulo its maximal divisible
subgroup. Note that $\mathcal{W}$ is finite of order independent of $L$,
and is trivial if $v\not\in\Sigma$.  The claim now follows from the 
identification
$$H^1(E_v,\bT)/H^1_\unr(E_v,\bT)\iso\mil\bigoplus_{w\mid v}
H^1(E_{k,w},T_\gp(A))/H^1_\unr(E_{k,w},T_\gp(A))$$
of Shapiro's lemma, together with the perfect pairing
$$H^1(E_v,\bT)/H^1_\unr(E_v,\bT)\times H^1_\unr(E_v,\bW)
\map{}\D_\gp$$ of Tate local duality.
 \end{proof}

\begin{Def} Following \cite{coates-greenberg},
we define the \emph{ordinary} Selmer structures $\sel_\ord$ on
$\bT$ and $\bW$ by
\begin{eqnarray*}
H^1_{\sel_\ord}(E_v,\bT) &=& \left\{\begin{array}{ll}
H^1(E_v,\bT^+)&\mathrm{if\ }v\mid \p\\
H^1(E_v,\bT)&\mathrm{else}\end{array}\right.\\
H^1_{\sel_\ord}(E_v,\bW) &=& \left\{\begin{array}{ll}
H^1(E_v,\bW^+)&\mathrm{if\ }v\mid\p\\
0&\mathrm{else}\end{array}\right.
\end{eqnarray*} and remark that these local conditions are everywhere
exact orthogonal complements under the local Tate pairing
$H^1(E_v,\bT)\times H^1(E_v,\bW)\map{}\D_\gp$.
\end{Def}

By standard results, the Selmer groups 
$H^1_{\sel_\ord}(E,\bT)$ and $H^1_{\sel_\ord}(E,\bW)$ are finitely and 
cofinitely generated, respectively, as $\Lambda$-modules.
Let $$X=\Hom_{\co_\gp}(H^1_{\sel_\ord}(E,\bW),\D_\gp),$$ and let 
$X_\tors\subset X$ be the $\Lambda$-torsion submodule.

\begin{Rem}
By the main result of \cite{coates-greenberg}, Shapiro's lemma
identifies the module $X$ with the $X$ defined in the Introduction 
(exactly, not just up to pseudo-isomorphism!). Similarly,
$H^1_{\sel_\ord}(E,\bT)$ is identified with the module
$S_{\gp,\infty}$ of the Introduction.
\end{Rem}

\begin{Prop}\label{lambda torsion}
The $\Lambda$-module $H^1_{\sel_\ord}(E,\bT)$ is torsion-free.
\end{Prop}
\begin{proof}  It suffices to show that $H^1(E^\Sigma/E,\bT)$
is torsion-free, and we imitate the
method of \cite{pr87}.  Let $\Lambda_n=\co_\gp[\Gal(E_n/E)]$ and set
$$X_n=\Hom_{\co_\gp}(H^1(E^\Sigma/E_n,A[\gp^\infty]),\D_\gp).$$
Using $A(E_\infty)[\gp]=0$ we have a canonical isomorphism
$$H^1(E^\Sigma/E_n,T_\gp(A))\iso\Hom_{\co_\gp}(X_n,\co_\gp),$$
and the map $\Lambda_n\map{}\co_\gp$ defined by extracting 
the coefficient of the
neutral element of $\Gal(E_n/E)$ induces an isomorphism
$$\Hom_{\Lambda_n}(X_n,\Lambda_n)\iso \Hom_{\co_\gp}(X_n,\co_\gp).$$
Using Shapiro's lemma, we obtain in the limit an isomorphism
$$H^1(E^\Sigma/E,\bT)\iso\Hom_\Lambda(\mil X_n,\Lambda),$$
which proves the claim.
\end{proof}

\begin{Def}
By a \emph{specialization} of $\Lambda$, we mean the ring of integers $S$ 
of a finite extension of $\co_\gp$, together with a homomorphism of 
$\co_\gp$-algebras $\phi:\Lambda\map{}S$ with finite cokernel. 
If $\phi:\Lambda\map{}S$ is a specialization, we define the 
\emph{dual specialization} $\phi^*:\Lambda\map{}S^*$ by $S^*=S$ and
$\phi^*=\phi\circ\iota$.  The maximal ideal of $S$ is denoted $\gm=\gm_S$.
\end{Def}

We view all specializations as taking values in a fixed algebraic
closure of $\Phi_\gp$.
For any specialization $\phi:\Lambda\map{}S$, we let $\Phi_S$ be the
field of fractions of $S$, and set $\D_S=\Phi_S/S$.
Furthermore, we define $S$-modules
$$T_S=\bT\otimes_\Lambda S\hspace{1cm} V_S=T_S\otimes_S\Phi_S
\hspace{1cm} W_S=V_S/T_S,$$ and for each prime of $E$ above $\p$,
$$T^\pm_S=\bT^\pm\otimes_\Lambda S\hspace{1cm} V^\pm_S=T^\pm_
S\otimes_S\Phi_S\hspace{1cm} W^\pm_S=V^\pm_S/T^\pm_S.$$
We regard $T_S$ as a $G_E$ module, with $G_E$ acting trivially on 
$S$.  Alternatively, we may identify $T_S\iso T_\gp(A)\otimes_{\co_\gp} S$
with $G_E$ acting on the second factor by $s^\sigma=\phi(\sigma^{-1})s$.

If we let $e_\gp$ denote the pairing of Lemma \ref{prelim lemma two}
and identify both $T_S$ and $T_{S^*}$ with
$T_\gp(A)\otimes_{\co_\gp}S$ as $S$-modules (which defines two distinct
$\Lambda$ and $G_E$-module structures on $T_\gp(A)\otimes_{\co_\gp}S$), 
then the pairing \begin{equation}\label{special pairing}
T_S\times T_{S^*}\map{}S(1)
\end{equation} defined by
$(t_0\otimes s_0, t_1\otimes s_1)\mapsto s_0s_1\cdot e_\gp(t_0,t_1)$
is perfect, $S$-bilinear, $G_E$-equivariant, and satisfies
$(\lambda x,y)=(x,\lambda^\iota y)$ for $\lambda\in\Lambda$.
If $\mathfrak{d}$ is a generator for the inverse different 
of $\Phi_S/\Phi_\gp$, then the composition
$$S\map{\mathfrak{d}}\mathfrak{d}S\map{\mathrm{Trace}}
\co_\gp,$$ together with the pairing (\ref{special pairing}),
gives perfect pairings
$$T_S\times T_{S^*}\map{}\co_\gp(1)\hspace{1cm}
T_S\times W_{S^*}\map{}\D_\gp(1).$$
Dualizing the map $\bT\map{}T_{S^*}$, and using the pairing
(\ref{iwasawa pairing}), we obtain a $G_E$-equivariant map of 
$\Lambda$-modules $W_S\map{}\bW$.

\begin{Def}
If $S$ is a specialization of $\Lambda$, we define a Selmer structure 
$\sel_S$ on $V_S$ by
$$H^1_{\sel_S}(E_v,V_S) = \left\{\begin{array}{ll}
H^1(E_v,V_S^+)&\mathrm{if\ }v\mid\p\\
H^1_\unr(E_v,V_S)&\mathrm{else}\end{array}\right.$$ 
and propagate this to Selmer structures on $T_S$ and $W_S$.
\end{Def}

\begin{Prop}\label{flach}
For every specialization $S$ of $\Lambda$, there is an $S$-bilinear pairing
$$H^1_{\sel_S}(E,W_S)\times H^1_{\sel_{S^*}}(E,W_{S^*})\map{}\D_S$$
whose kernels on either side are the submodules of $S$-divisible
elements.
\end{Prop}
\begin{proof} This follows from the main result of \cite{flach}, the construction
of a generalized Cassels-Tate pairing. \end{proof}

\begin{Lem}\label{special hypotheses}
For any specialization $\phi:S\map{}\Lambda$, the module $T_S$ and the 
Selmer structure $\sel_S$ on $T_S$ satisfy hypotheses {\bf H1--H5}
of Section \ref{main bound}.
\end{Lem}
\begin{proof}  Let $L=E(A[\gp^\infty])$ and $L_\infty=LE_\infty$ and  consider
the inflation-restriction sequence
$$0\map{}H^1(L/E,A[\gp])\map{}
H^1(L_\infty/E,A[\gp])\map{}\Hom(\Gal(L_\infty/L),A[\gp])^{\Gal(L/E)}.$$
The final term is $0$, since we are assuming condition 
(\ref{galois image}) of the Introduction.  The term
$H^1(L/E,A[\gp])$ is also zero, by the proof of Theorem
\ref{little money}, and so {\bf H1} holds.  For {\bf H2}, we may identify
$T_S\iso T_\gp(A)\otimes_{\co_\gp}S$ with $G_E$ acting on 
$S$ via $G_E\map{}\Lambda^\times\map{\iota}\Lambda^\times\map{}S^\times$.
In particular, the action of $G_E$ on the residual representation of
$S$ is trivial.  The residual representaion of $T_S$ is therefore
isomorphic to $A[\gp]\otimes S$, with $G_E$ now acting only on the first
factor.  

For hypothesis {\bf H3}, we again identify $T_S$
with $T_\gp(A)\otimes_{\co_\gp} S$.  Let $e_\gp$
be the pairing of Lemma \ref{prelim lemma two}, and define a pairing
$$T_\gp(A)\otimes_{\co_\gp}S\times T_\gp(A)\otimes_{\co_\gp}S\map{}S(1)$$ 
by $(t_0\otimes s_0, t_1\otimes s_1)\mapsto s_0s_1\cdot e_\gp(t_0,t_1^\tau)$.
It is trivial to verify that this pairing has the desired properties.
Hypotheses {\bf H4} and {\bf H5} follow easily from the definition of
$\sel_S$.
 \end{proof}

In the remainder of this section we prove some technical lemmas needed 
in the next section.

\begin{Lem}\label{ordinary reduction}
Let $\phi:\Lambda\map{}S$ be a specialization with kernel $I$, and
let $\q$ be a prime of $E$ above $\p$.  The cokernel of the natural map
$$H^1(E_\q,\bT^+)\map{}H^1(E_\q,\bT^+/I\bT^+)$$
is finite with order bounded by a constant which depends only
on $\mathrm{rank}_{\co_\gp}(S)$.
\end{Lem}
\begin{proof} 
We extend scalars to $S$:  let $\Lambda'=\Lambda\otimes_{\co_\gp}S$
and extend $\phi$ to a surjective $S$-module map 
$\phi':\Lambda'\map{}S$.  Fix an identification
$\Lambda'\iso S[[s_1,\ldots,s_g]]$ and
define $\alpha_i=\phi(s_i)\in S$. Let
$I$ be the ideal of $\Lambda'$ generated by all $(s_i-\alpha_i)$.
Consider the map \begin{equation}\label{extension quotient}
H^1(E_\q,T_\gp(A)^+\otimes\Lambda')\map{}
H^1(E_\q,T_\gp(A)^+\otimes(\Lambda'/I)).\end{equation}
If $I_r\subset\Lambda'$ is the ideal generated by
$s_i-\alpha_i$ for $1\le i\le r$, then the cohomology
of $$0\map{}\Lambda'/I_r\map{s_{r+1}-\alpha_{r+1}}
\Lambda'/I_r\map{}\Lambda'/I_{r+1}\map{}0$$ tensored
(over $\co_\gp$) with $T_\gp(A)^+$, yields exactness of
\begin{eqnarray*}\lefteqn{H^1(E_\q,T_\gp(A)^+\otimes(\Lambda'/I_r))\map{}
H^1(E_\q,T_\gp(A)^+\otimes(\Lambda'/I_{r+1}))  }\hspace{5cm}\\ &  &\map{}
H^2(E_\q,T_\gp(A)^+\otimes(\Lambda'/I_r)).\end{eqnarray*}
By local duality and Shapiro's lemma the final term is dual to the
$I_r$-torsion submodule of
$$\dlim \oplus_{w|\q} H^0(E_{k,w},\tilde{A}[\gp^\infty])
\otimes_{\co_\gp} S.$$
Since $\q$ is totally ramified in $E_\infty$, this is 
equal to the $\gp$-power-torsion of $\tilde{A}$ rational over
the residue field of $E$ at $\q$ (tensored with $S$), 
which is clearly finite.  It follows
that the cokernel of (\ref{extension quotient}) is finite and bounded
by a constant depending only on $\mathrm{rank}_{\co_\gp}(S)$.
This map, however,
is exactly the map obtained by tensoring the map in the statement of 
the lemma (over $\co_\gp$) with $S$, and the claim follows. \end{proof}

\begin{Lem}\label{local control}
Let $\phi:\Lambda\map{}S$ be a specialization with kernel $I$.
For every place $v$ of $E$ the maps $\bT/I\bT\map{}T_S$ and
$W_S\map{}\bW[I]$ 
induce $\Lambda$-module maps
$$H^1_{\sel_\ord}(E_v,\bT/I\bT)\map{}H^1_{\sel_S}(E_v,T_S)$$
$$H^1_{\sel_S}(E_v,W_S)\map{}H^1_{\sel_\ord}(E_v,\bW[I]),$$
where the Selmer structure $\sel_\ord$ on $\bT/I\bT$ is propagated from 
$\bT$, and similarly for $\bW[I]$.  
The kernels and cokernels of these maps
are finite and bounded by constants depending only 
$\mathrm{rank}_{\co_\gp}(S)$ and $[S:\phi(\Lambda)]$.
\end{Lem}
\begin{proof} 
Bounds on the kernel and cokernel of the first map, in the case 
in which $v$ is a divisor of $\p$, are exactly as in 
the proof of Lemma 2.2.7 of \cite{howard}, together with Lemma 
\ref{ordinary reduction} above.

Consider the case where $v$ does not divide $\p$.  
We first show that the natural map 
$H^1_\unr(E_v,\bT)\map{}
H^1_\unr(E_v,\bT/I\bT)$ is surjective.  Indeed, as $\Gal(E_v^\unr/E_v)$
has cohomological dimension one, it suffices to show that
$\bT^\inert\map{}(\bT/I\bT)^\inert$ is surjective, where $\inert$
is the inertia subgroup of $G_{E_v}$.  Identifying 
$\bT\iso T_\gp(A)\otimes\Lambda$, and using the fact that 
$T_\gp(A)$ is a flat $\co_\gp$-module and that $\inert$ acts
trivially on $\Lambda$, this is equivalent to
the surjectivity of 
$T_\gp(A)^\inert\otimes\Lambda\map{}T_\gp(A)^\inert\otimes\Lambda/I,$
which is clear.

 Now applying Lemma \ref{unramified}, we see that 
\begin{equation}\label{unramified comparison}
H^1_\unr(E_v,\bT/I\bT)\subset H^1_{\sel_\ord}(E_v,\bT/I\bT)
\end{equation}
with finite index, and equality holds if $v\not\in\Sigma$.
Furthermore, the index depends only on 
$\mathrm{rank}_{\co_\gp}(S)$ and not on $\phi$.
The quotient of $H^1(E_v,T_S)$ by $H^1_{\sel_S}(E_v,T_S)$
is torsion free, and it follows that the image of 
$$H^1_{\sel_\ord}(E_v,\bT/I\bT)\map{}H^1(E_v,T_S)$$
is contained in $H^1_{\sel_S}(E_v,T_S)$.

Exactly
as in the proof of Lemma 5.3.13 of \cite{mazur-rubin}, the kernel and
cokernel of the composition 
$$H^1_\unr(E_v,\bT/I\bT)\map{}H^1_\unr(E_v,T_S)\hookrightarrow
H^1_{\sel_S}(E_v,T_S)$$ are finite with bounds of the desired sort.
The claims concerning the kernel and cokernel of
first map of the lemma follow without difficulty.
The claims concerning the second map of the statement of the lemma
follow from  Tate local duality.
\end{proof}

\begin{Lem}\label{global control}
Let $\phi:\Lambda\map{}S$ be a specialization with kernel $I$.
The maps $\bT/I\bT\map{}T_S$ and $W_S\map{}\bW[I]$ 
induce $\Lambda$-module maps on global cohomology
$$H^1_{\sel_\ord}(E,\bT/I\bT)\map{}H^1_{\sel_S}(E,T_S)$$
$$H^1_{\sel_S}(E,W_S)\map{}H^1_{\sel_\ord}(E,\bW[I]),$$
where the Selmer structure $\sel_\ord$ on $\bT/I\bT$ is propagated from 
$\bT$, and similarly for $\bW[I]$.  
The kernels and cokernels of these maps
are finite and bounded by constants depending only on
$\mathrm{rank}_{\co_\gp}(S)$ and $[S:\phi(\Lambda)]$.
\end{Lem}
\begin{proof} 
This can be deduced from the proof of Proposition 5.3.14 of 
\cite{mazur-rubin},  once we show that $H^1(E,\bT)=H^1(E^\Sigma/E,\bT)$.
If $v\not\in\Sigma$ is a prime of $E$ 
which does not split completely in $E_\infty$,
then $H^1(E_v,\bT)=H^1_\unr(E_v,\bT)$ by \cite{rubin} Proposition B.3.4,
while if $v$ does splits completely we have
$$H^1(E_v,\bT)=H^1(E_v,T)\otimes_{\co_\gp}\Lambda
=H^1_\unr(E_v,T)\otimes_{\co_\gp}\Lambda=H^1_\unr(E_v,\bT).$$ 
By the proof of \cite{rubin} Proposition 1.6.8,
$H^1_\unr(E_v,T\otimes\Q_p)=0$, and so local Tate duality 
and the Weil pairing force $H^1(E_v,T\otimes\Q_p)=0$.
From \cite{rubin} Lemma 1.3.5 it then follows that 
$H^1(E_v,T)=H^1_\unr(E_v,T)$.
\end{proof}

\begin{Lem}\label{global control two}
For any ideal $I\subset\Lambda$, the inclusion $\bW[I]\map{}\bW$
induces an isomorphism $$H^1_{\sel_\ord}(E,\bW[I])\iso
H^1_{\sel_\ord}(E,\bW)[I].$$
\end{Lem}
\begin{proof} See Lemma 3.5.3 of \cite{mazur-rubin}. \end{proof}


\subsection{Choosing specializations}
\label{choosing specs}


In this section we construct sequences of specializations with nice
properties.  The reader is advised to read the statements of 
Proposition \ref{choosing specializations} and Corollary
\ref{lambda ranks}, and proceed directly
to Section \ref{main conjecture section}.

Throughout this section we fix  a height-one prime $\gq\not=\gp\Lambda$
of $\Lambda$ and some nonzero $c\in H^1_{\sel_\ord}(E,\bT)$. 
Set $$\mathbf{L}_c=H^1_{\sel_\ord}(E,\bT)/\Lambda c.$$
By convention the characteristic ideal of a $\Lambda$-module of positive
rank is zero, and the order at $\gq$ of the zero ideal is infinite.
Let $\phi:\Lambda\map{}S$ be a specialization. 
By Lemma \ref{global control}, the maps $\bT\map{}T_S$ 
and $W_S\map{}\bW$ induce maps, still denoted $\phi$, on Selmer modules
$$H^1_{\sel_\ord}(E,\bT)\map{}H^1_{\sel_S}(E,T_S)
\hspace{1cm} H^1_{\sel_S}(E,W_S)\map{}H^1_{\sel_\ord}(E,\bW).$$
Our goal is to exhibit many specializations for which we have
careful control over the kernels and cokernels of these maps.

For any specialization $S$, Proposition \ref{structure} and Lemma
\ref{special hypotheses} show that the quotient of
$H^1_{\sel_S}(E,W_S)$ by its maximal $S$-divisible submodule has the
form $M_S\oplus M_S$.

Let $Q\in\Lambda$ generate the ideal $\gq$, and consider the 
following properties:

\newcounter{spechyp}
\begin{list}{\bfseries\upshape Sp\arabic{spechyp}:}
{\usecounter{spechyp}}
\item The $\Lambda$-rank of $H^1_{\sel_\ord}(E,\bT)$ is equal to
the $S$-rank of $H^1_{\sel_S}(E,T_S)$.
\item The $\Lambda$-rank
of $X$ is equal to the $S$-corank of $H^1_{\sel_S}(E,W_S).$
\item With $M_S$ as above
$$2\cdot\mathrm{length}_S (M_S)
=\ord_\gq\big(\mathrm{char}(X_{\Lambda-\tors})\big)\cdot 
\mathrm{length}_S\big(S/\phi(Q)S\big).$$
\item The equality
$$\mathrm{length}_S \big(H^1_{\sel_S}(E,T_S)/S\cdot\phi(c)\big)
=\ord_\gq(\mathrm{char}(\mathbf{L}_c))\cdot 
\mathrm{length}_S\big(S/\phi(Q)S\big)$$ holds
(we include the case where both sides are infinite),
\item the image of $c$ in $H^1_{\sel_S}(E,T_S)$ is nonzero.
\end{list}

\begin{Def}
A sequence of specializations $\phi_i:\Lambda\map{}S$ (with $S$ independent
of $i$) is said to converge to $\gq=Q\Lambda$ if 
\begin{enumerate}
\item $\phi_i(Q)\to 0$, but $\phi_i(Q)\not=0$ for all $i$,
\item {\bf Sp1,2,5} hold for every $i$,
\item the equalities {\bf Sp3,4} hold up to $O(1)$ as $i$ varies,
\item the maps $H^1_{\sel_\ord}(E,\bT)\otimes_\Lambda S\map{}
H^1_{\sel_S}(E,T_S)$ have finite kernels and cokernels, bounded as $i$
varies.
\end{enumerate}
\end{Def}

\begin{Rem}
The reader should keep in mind that the notation 
$T_S$, $W_S$, $\sel_S$, $?\otimes_\Lambda S$, and so on
is slightly abusive, since these objects depend not only on the ring $S$
(which will typically remain fixed),
but on its structure as a $\Lambda$-algebra (which will typically vary).
We will continue to supress this dependence from the notation.
\end{Rem}

This section is devoted to the proof of the following proposition:

\begin{Prop}\label{choosing specializations}
There exists a sequence of specializations converging to $\gq$.
\end{Prop}

The following result  can be deduced from Nekov\'{a}\v{r}'s general
theory of Selmer complexes \cite{nekovar},
(in particular the ``duality diagram'' of section 0.13), 
but is also a trivial consequence of the proposition above.

\begin{Cor}\label{lambda ranks}
The modules $H^1_{\sel_\ord}(E,\bT)$ and $X$ have the same $\Lambda$-rank.
\end{Cor}
\begin{proof} If $S$ is a specialization with 
uniformizer $\pi$, the Selmer group 
$H^1_{\sel_S}(E,T_S)$ is isomorphic to the $\pi$-adic Tate module
of $H^1_{\sel_S}(E,W_S)$ (using Lemmas 3.5.4 and 3.7.1 
of \cite{mazur-rubin}).  The claim therefore follows 
from the existence of a single specialization for
which properties {\bf Sp1} and {\bf Sp2} hold.
 \end{proof}

\begin{Def}
Let $J\subset \Lambda$ be an ideal, $\phi:\Lambda\map{}S$ a specialization,
let $$d(S,J)=\mathrm{min}\big(\{v(\phi(\lambda))\mid \lambda\in J\}\big),$$
where $v$ is the normalized valuation on $S$. 
We define the \emph{distance} from $S$ to $J$ to be $p^{-d(S,J)}$
(including the case $d(S,J)=\infty$, in which case the distance is zero).
\end{Def}

\begin{Rem}
The geometric intuition behind the definition is as follows: if one were
to replace $\Lambda$ by a polynomial ring over $\bar{\Q}_p$, then $J$
cuts out an algebraic subset $V(J)$ of affine space.  
Geometrically, a specialization
is a point in affine space, and the distance function defined above measures
the $p$-adic distance from this point to $V(J)$.
\end{Rem}

\begin{Lem}\label{sequence lemma}
There is a height-two ideal $J\subset\Lambda$ with the following property:
if  $\{\phi_i:\Lambda\map{}S\}$ is a sequence of specializations 
(with $S$ fixed), such that as $i$ varies
\begin{enumerate}
\item $[S:\phi_i(\Lambda)]$ is bounded above,
\item the distance from $\gq$ to $\phi_i$ converges to 
(but is never equal to) zero,
\item  the distance from $J$ to $\phi_i$ is bounded away from zero,
\item the maps $H^1_{\sel_\ord}(E,\bT)\otimes_\Lambda S\map{}
H^1_{\sel_S}(E,T_S)$ have finite and uniformly bounded kernels and cokernels, 
\end{enumerate}
then $\phi_i$ converges to $\gq$.
\end{Lem}
\begin{proof} 
Fix pseudo-isomorphisms
\begin{eqnarray}
\label{first spec}
H^1_{\sel_\ord}(E,\bT)  &\map{}&  \Lambda^{r_1}\\
\nonumber
X  &\map{}&  \Lambda^{r_0}\oplus B\oplus C\\
\nonumber 
\mathbf{L}_c &\map{}& \Lambda^{r_1-1}\oplus B'\oplus C'
\end{eqnarray}
such that $B$ and $C$ are direct sums of torsion, cyclic $\Lambda$-modules,
with the characteristic ideal of $B$ a power of $\gq$,
the characteristic ideal of $C$ prime to $\gq$, and similarly
for $B'$ and $C'$.

Let $J\subset \Lambda$ be any height two ideal such that
\begin{enumerate}
\item $J$ annihilates the kernels and cokernels of the above 
pseudo-isomorphisms
\item $J\subset\gq+\mathrm{char}(C)$
\item $J\subset\gq+\mathrm{char}(C')$
\end{enumerate}
and let $\phi_i:\Lambda\map{}S$ 
be a sequence of specializations satisfying the 
hypotheses of the Lemma.  The condition that the distance from $\phi_i$
to $J$ is bounded below guarantees that the maps obtained by tensoring
the maps of (\ref{first spec}) with 
$\phi_i$ have finite kernels and cokernels, bounded as $i$ varies.
Indeed, if $U$ and $V$ denote the kernel and cokernel of 
any one of (\ref{first spec}), the kernel and cokernel of the map
tensored with $S$ are controlled by $U\otimes S$, $V\otimes S$, and
$\mathrm{Tor}^1_\Lambda(V,S)$.  The number of generators of these
modules does not depend on the map $\phi_i$, and the assumption that 
$\phi_i$ is bounded away from $J$ gives a nonzero element of $S$, independent
of $i$, which annihilates all of these modules.

It follows that the $S$-rank of $H^1_{\sel_\ord}(E,\bT)\otimes S$
is equal to $r_1$, the $\Lambda$-rank of $H^1_{\sel_\ord}(E,\bT)$.
This, together with the fourth hypothesis of the Lemma, verifies
property {\bf Sp1}.
The second condition defining $J$, together with the assumption that
the distance from $\phi_i$ to $\gq$ goes to zero, implies that 
the distance from $\phi_i$ to $\mathrm{char}(C)$ is bounded below.
Therefore $C\otimes S$ is finite and bounded as $i$ varies.
Writing 
$B\iso\bigoplus\Lambda/\gq^{e_k},$
we have that $B\otimes S\iso \bigoplus S/\phi_i(\gq)^{e_k}S$
is a torsion $S$-module and
$$\mathrm{length}_S(B\otimes S)=(\sum e_k)
\mathrm{length}_S(S/\phi_i(\gq)S).$$
Thus $X\otimes S$ is an $S$-module of rank $r_0$ whose torsion
submodule has length $$\ord_\gq\big(\mathrm{char}
(X_{\Lambda-\tors})\big)\cdot 
\mathrm{length}_S\big(S/\phi_i(\gq)S\big)$$ up to $O(1)$ as $i$ varies.
Applying Lemmas \ref{global control} and \ref{global control two} and 
dualizing, we see that {\bf Sp2} holds, and that the equality 
{\bf Sp3} holds up to $O(1)$ as $i$ varies.

Exactly as above, $\mathbf{L}_c\otimes S$ is an $S$-module
of rank $r_1-1$  and length
$$\ord_\gq\big(\mathrm{char}
(\mathbf{L}_{c})\big)\cdot 
\mathrm{length}_S\big(S/\phi_i(\gq)S\big)$$
up to $O(1)$ as $i$ varies.
In the exact sequence 
$$(\Lambda\cdot c)\otimes S\map{}H^1_{\sel_\ord}(E,\bT)\otimes S
\map{}\mathbf{L}_c\otimes S\map{}0$$
the second and third modules have $S$-rank $r_1$ and $r_1-1$, respectively,
and so the first arrow must be an injection.
The remaining properties now follow from the fourth 
hypothesis of the Lemma.
 \end{proof}

The difficulty lies in producing a sequence of specializations for
which hypothesis (d) holds.

\begin{Lem}\label{cohomology torsion}
Let $S$ be the ring of integers of a finite extension of $\Phi_\gp$,
and let $M$ be a finitely generated $S[[x_1,\ldots,x_r]]$-module.  
For all but finitely many $\alpha\in \gm=\gm_S$, 
$M[x_r-\alpha]$ is a torsion $S$-module and a pseudo-null 
$S[[x_1,\ldots,x_r]]$-module.
\end{Lem}
\begin{proof} The module $M[x_r-\alpha]$ is pseudo-null whenever $x_r-\alpha$
does not divide the characteristic ideal of the 
$S[[x_1,\ldots,x_r]]$-torsion submodule of $M$, and so this
condition causes no difficulty.

For the $S$-torsion condition, first suppose that $M$ has no $S$-torsion.  
Recall that a prime ideal
of $S[[x_1,\ldots,x_r]]$ is said to be an \emph{associated prime of $M$} if it
is the exact annihilator of some $m\in M$.  By the theory of primary
decomposition, 
$M$ has only finitely many associated primes, and $M[x_r-\alpha]$
is trivial unless $x_r-\alpha$ is contained in some associated prime.
If $\gq$ is an associated prime of $M$ with $(x_r-\alpha)$ and $(x_r-\beta)$
both contained in $\gq$, then $\alpha-\beta$ is contained in $\gq$.  
By the definition of an associated prime, $M$ has a submodule isomorphic to
$S[[x_1,\ldots,x_r]]/\gq$, and so $S[[x_1,\ldots,x_r]]/\gq$ 
can have no $S$-torsion.  Therefore
$\alpha=\beta$, and so for every associated prime there is at most
one $\alpha$ for which $x_r-\alpha$ is contained in that prime.  

The case of arbitrary $M$ now follows easily from the exactness of
$$0\map{}M_{S-\tors}[x_r-\alpha]\map{}M[x_r-\alpha]\map{}
(M/M_{S-\tors})[x_r-\alpha].$$
\end{proof}

We are now ready to begin the proof of Proposition 
\ref{choosing specializations}.  Let $J$ be as in Lemma \ref{sequence lemma},
and let $\phi:\Lambda\map{}S$ be a specialization such that 
$\phi(Q)=0$ and such that the distance from $S$ to $J$ is nonzero.
We fix an identification $\Lambda\iso \co_\gp[[x_1,\ldots,x_g]]$
in such a way that $Q(x_1,b_2\ldots,b_{g})$ is not identically
zero as a power series in $x_1$, where $b_i=\phi(x_i)$ for 
$1\le i\le g$.  By Hensel's lemma, for every $i$ there is an open 
neighborhood  $U_i\subset\gm$ of $b_i$,  such that for any $\beta_i\in U_i$, 
the subring $\co_\gp[\beta_i]\subset S$ is equal to $\co_\gp[b_i]$. 
Hence if $\beta\in U_1\times\cdots\times U_g$, the map 
$\phi^\beta:\Lambda\map{}S$ taking $x_i\mapsto\beta_i$,
determines a specialization of $\Lambda$.

For $0\le r\le g$ define $\Lambda_r=S[[x_1,\ldots,x_r]]$.  For $\beta\in U$,
sending $x_i\mapsto \beta_i$ for $r+1\le i\le g$ determines a map 
$\Lambda\map{}\Lambda_r$.  When we view $\Lambda_r$ as a $\Lambda$-algebra in
this way, we will write $\Lambda_r^\beta$ to emphasize the dependence on
$\beta$. 
Composing these maps with the character 
$G_E\map{}\Lambda^\times\map{\iota}\Lambda^\times$,
we obtain characters  
$\chi_r:\Gal(E^\Sigma/E)\map{}(\Lambda_r^\beta)^\times$.
Set $$\bT^\beta_r=T_\gp(A)\otimes_{\co_\gp}\Lambda^\beta_r$$ with 
$\Gal(E^\Sigma/E)$ 
acting on both factors, and for each prime $\q$ dividing $\p$ define
$\bT_r^{\beta\pm}=T_\gp(A)^\pm\otimes\Lambda^\beta_r$.  
Sending $x_r\mapsto \beta_r$ determines a 
$\Lambda$-algebra map $\Lambda^\beta_{r}\map{}\Lambda^\beta_{r-1}$,
which induces the  exact sequence
$$0\map{}\bT_r^\beta\map{x_r-\beta_r}\bT_r^\beta\map{}\bT_{r-1}^\beta\map{}0$$
and similarly with $\bT$ replaced by $\bT^\pm$.
Set $$L_r^\beta=\bigoplus_{\q\mid\p}H^1(E_q,\bT_r^{\beta-})
\oplus\bigoplus_{v}H^1(\inert_v,\bT_r^\beta)$$
where the second sum is over all $v\in\Sigma$ not dividing $\p$, and
$\inert_v$ is the inertia subgroup of $\Gal(\bar{E}_v/E_v)$.
We define generalized Selmer groups $H_r^\beta$ by the exactness of
$$0\map{} H_r^\beta\map{}H^1(E^\Sigma/E,\bT_r^\beta)\map{}L_r^\beta.$$

\begin{Rem}\label{beta coords}
The $\Lambda$-algebra $\Lambda^\beta_r$ depends only on the coordinates 
$\beta_i$ with $ i>r$, and similarly for $\bT_r^\beta$ and $H_r^\beta$.
\end{Rem}

\begin{Lem}\label{inductive step}
Fix $1\le r\le g$, and suppose we are given $\beta_i\in\gm$ for $i>r$.
For all but finitely many $\beta_r\in\gm$, the map $\bT^\beta_{r}\map{}
\bT^\beta_{r-1}$ induces a map
$$H^\beta_r\otimes_{\Lambda^\beta_r}\Lambda^\beta_{r-1}\map{}
H^\beta_{r-1}$$ with $S$-torsion kernel and cokernel. If $r=1$, there
is a subset $V\subset\gm$ of finite complement such that kernel
and cokernel are finite and bounded as $\beta_1$ ranges over $V$.
\end{Lem}
\begin{proof} 
Let $M_r^\beta$ denote the image of
$H^1(E^\Sigma/E,\bT_r^\beta)\map{}L_r^\beta$ and consider the 
commutative diagram
$$\xymatrix{
0\ar[r]&H^\beta_{r}\ar[r]\ar[d]^{x_r-\beta_r}
&H^1(E^\Sigma/E,\bT_r^\beta)\ar[r]\ar[d]^{x_r-\beta_r}
&M_r^\beta\ar[d]^{x_r-\beta_r}\ar[r]&0\\
0\ar[r]&H^\beta_{r}\ar[r]\ar[d]^\mu
&H^1(E^\Sigma/E,\bT_r^\beta)\ar[r]\ar[d]^\nu
&M_r^\beta\ar[d]^\xi\ar[r]&0\\
0\ar[r]&H^\beta_{r-1}\ar[r]&H^1(E^\Sigma/E,\bT_{r-1}^\beta)\ar[r]
&M_{r-1}^\beta\ar[r]&0
}$$
in which all rows and the middle column are exact.
Viewing this as an exact sequence of vertical complexes and
taking cohomology, the kernel of
$$H_r^\beta/(x_r-\beta_r)H_r^\beta\map{\mu}H^\beta_{r-1}$$
is isomorphic to the cokernel of 
$$H^1(E^\Sigma/E,\bT_r^\beta)[x_r-\beta_r]\map{}
M^\beta_r[x_r-\beta_r],$$
and so is $S$-torsion for all
but finitely many choices of $\beta_r$ by Lemma \ref{cohomology torsion}.
Furthermore, in the case $r=1$,
$M^\beta_r[x_1-\beta_1]$ is bounded by the order
of the maximal pseudo-null (hence finite) submodule of $M_r^\beta$.

The cokernel of $\mu$ is bounded in terms of the 
cokernel of $\nu$ and the kernel
of \begin{equation}\label{some stupid map}
M_r^\beta/(x_r-\beta_r)M_r^\beta\map{\xi}M^\beta_{r-1}.\end{equation}
The cokernel of $\nu$ is isomorphic to the $x_r-\beta_r$ torsion
in $H^2(E^\Sigma/E,\bT^\beta_r)$, which is again controlled by
Lemma \ref{cohomology torsion}.  
The middle column of  $$\xymatrix{
0\ar[r]&M^\beta_{r}\ar[r]\ar[d]^{x_r-\beta_r}
&L^\beta_r\ar[r]\ar[d]^{x_r-\beta_r}
&M_r^\beta/L_r^\beta\ar[d]^{x_r-\beta_r}\ar[r]&0\\
0\ar[r]&M^\beta_{r}\ar[r]\ar[d]
&L^\beta_r\ar[r]\ar[d]
&M_r^\beta/L_r^\beta\ar[d]\ar[r]&0\\
0\ar[r]&M^\beta_{r-1}\ar[r]&L^\beta_{r-1}\ar[r]
&M_{r-1}^\beta/L_{r-1}^\beta\ar[r]&0
}$$ is exact, and as above the kernel of
(\ref{some stupid map}) is isomorphic to the cokernel
of $$L_r^\beta[x_r-\beta_r]\map{}(M_r^\beta/L_r^\beta)[x_r-\beta_r],$$
Again this is controlled by Lemma \ref{cohomology torsion}.
 \end{proof}

\begin{Lem}
For $0\le r\le g$, there is a dense subset 
$\Delta_r\subset U_{r+1}\times\cdots\times U_g$ 
such that for $\beta\in U_1\times\cdots\times U_r\times \Delta_r$,
the map $\bT_g^\beta\map{}\bT_r^\beta$ induces a map
$$H^\beta_g\otimes_{\Lambda_g^\beta}\Lambda_r^\beta  
\map{}H^\beta_r$$ with $S$-torsion kernel and cokernel.
\end{Lem}
\begin{proof} An easy induction using the preceeding lemma. \end{proof}

The power series $Q(x_1,b_2,\ldots,b_g)\in S[[x_1]]$ is not identically
zero by assumption and has a zero in $\gm$, namely $b_1$.  Using
the Weierstrass preparation theorem and Hensel's lemma, we see that
if we replace $b_2,\ldots,b_g$ by sufficiently nearby points 
$\beta_2,\ldots,\beta_g$, then $Q(x_1,\beta_2,\ldots,\beta_g)$
has a zero, $\beta_1$, which is as close as we like to $b_1$.
Replacing $b$ by a nearby solution to $Q(x_1,\ldots,x_g)=0$,
we henceforth assume that $b\in U_1\times \Delta_1$.
Define a sequence of specializations $\phi_i:\Lambda\map{}S$
by $$\phi_i(x_r)=\left\{\begin{array}{ll}b_1+p^i &\mathrm{if\ }r=1\\
b_r &\mathrm{if\ }r>1.\end{array}\right.$$
We shall always assume that $i$ is chosen large enough that 
$b_1+p^i\in U_1$ and that $\phi_i(Q)\not=0$, the second being possible
since, by Weierstrass preparation, the function $Q(x_1,b_2,\ldots,b_g)$ 
has only finitely many zeros in $\gm$.  Also, the distance from $\phi_i$
to $J$ converges to the (nonzero) distance from $\phi$ to $J$, 
and so the sequence satisfies conditions (a)--(c) of Lemma
\ref{sequence lemma}.  We let $\beta(i)\in\gm^g$ be the
point with coordinates $\phi_i(x_1),\ldots,\phi_i(x_g)$, and set ourselves
to the task of showing that the sequence $\phi_i$ satisfies 
condition (d) of Lemma \ref{sequence lemma}.

\begin{Lem}
Let $\beta=\beta(i)$ for $i\gg 0$. The map $\bT_g^\beta\map{}\bT_0^\beta$
induces a map $$H^\beta_g\otimes_{\Lambda_g^\beta}\Lambda_0^\beta  
\map{}H^\beta_0$$ with finite kernel and cokernel, bounded as $i$ varies.
\end{Lem}
\begin{proof} By choice of $\beta_2,\ldots,\beta_g$ 
(which do not vary with $i$), 
the map $$H_g^\beta\otimes_{\Lambda_g^\beta}\Lambda_1^\beta\map{}H_1^\beta$$
has $S$-torsion kernel and cokernel.  Furthermore, by Remark
\ref{beta coords}, the number of generators (as $\Lambda^\beta_1$-modules)
and the annihilators (as $S$-modules) of the kernel 
and cokernel do not vary with $i$.
Consequently, tensoring this map with $\Lambda_0^\beta\iso S$,
the kernel and cokernel of 
$$H_g^\beta\otimes_{\Lambda_g^\beta}\Lambda_0^\beta\map{}
H_1^\beta\otimes_{\Lambda_1^\beta}\Lambda_0^\beta$$
are finite and bounded as $i$ varies.
By the final claim of Lemma \ref{inductive step}, the map
$$H_1^\beta\otimes_{\Lambda_1^\beta}\Lambda_0^\beta\map{}H^\beta_0$$
has finite kernel and cokernel, bounded as $i$ varies, and the claim is 
proven.  \end{proof}

\begin{Lem} Let $\beta=\beta(i)$, and $\phi=\phi_i:\Lambda\map{}S$ 
the associated specialization.
Identifying $\Lambda_0^\beta= S$, $H_0^\beta\subset H^1_{\sel_S}(E,T_S)$
with finite index, bounded as $i$ varies.
\end{Lem}
\begin{proof} Let $$L_S=\bigoplus_{\q\mid\p}H^1(E_q,V_S^-)
\oplus\bigoplus_{v}H^1(\inert_v,V_S)$$
where the second sum is over all $v\in\Sigma$ not dividing $\p$.
The map $T_S\map{}V_S$ induces a map $L_0^\beta\map{}L_S$,
and so from the definitions we have the commutative diagram with exact
rows 
$$\xymatrix{0\ar[r]&H_0^\beta\ar[r]\ar[d]&H^1(E^\Sigma/E,T_S)\ar[r]\ar[d]
&L_0^\beta\ar[d]\\
0\ar[r]&H^1_{\sel_S}(E,T_S)\ar[r]&H^1(E^\Sigma/E,T_S)\ar[r]&L_S}$$
and so it suffices to bound the kernel of $L_0^\beta\map{}L_S$.
If $\q$ divides $\p$, the kernel of $H^1(E_\q,T_S^-)\map{}H^1(E_\q,V_S^-)$
is bounded by the order of $H^0(E_\q,W_S^-)$.  Let $w$ be a place
of $E_\infty$ above $\q$.  It clearly suffices to bound the order
of $H^0(E_{\infty,w},W_S^-)$, but as a module over the 
absolute Galois group of $E_{\infty,w}$ we have 
$W_S^-\iso \tilde{A}[\gp^\infty]\otimes_{\co_\gp}S$, where $S$ has trivial 
Galois action and $\tilde{A}$ is the reduction of $A$ at $\q$.
It therefore suffices to show that
$H^0(E_{\infty,w},\tilde{A}[\gp^\infty])$ is finite.
Since $\tilde{A}[\gp^\infty]$ is cofree of rank one over $\co_\gp$,
if this is not the case then all of $\tilde{A}[\gp^\infty]$
is fixed by the Galois group of $E_{\infty,w}$.  But since
$\tilde{A}[\gp^\infty]$ is unramified over $E_\q$ and $E_{\infty,w}/E_\q$
is totally ramified, we must have
$$H^0(E_{\infty,w},\tilde{A}[\gp^\infty])=
H^0(E_\q,\tilde{A}[\gp^\infty]).$$
The right hand side is finite.

Similarly, the kernel of $H^1(\inert_v,T_S)\map{}H^1(\inert_v,V_S)$
is isomorphic to quotient of $H^0(\inert_v,W_S)$ by its maximal
$S$-divisible submodule, which is finite.  Since $\inert_v$ acts trivially
on $\Lambda$, it also acts trivially on $S\iso\Lambda^\beta_0$, 
regardless of the choice of $\beta$, and so the group $H^0(\inert_v,W_S)$
does not vary with $i$. \end{proof}

Let $\beta=\beta(i)$ and $\phi:\Lambda\map{}S$ the associated specialization.
Define $H_\ord^\unr(\bT)$ by exactness of 
$$0\map{}H_\ord^\unr(\bT)\map{}H^1_{\sel_\ord}(E,\bT)\map{}
\bigoplus_v H^1(E_v,\bT)/H^1_\unr(E_v,\bT),$$
where the second sum is over all $v\in\Sigma$ not dividing $\p$.
By Lemma \ref{unramified}, the map
$$H_\ord^\unr(\bT)\otimes_\Lambda S\map{} 
H^1_{\sel_\ord}(E,\bT)\otimes_\Lambda S$$ has finite kernel and cokernel,
bounded as $i$ varies.  We may identify $\bT\otimes_{\co_\gp} S\iso 
\bT_g^\beta$ (with $G_E$ acting trivially on $S$ in the left hand side),
and this identification induces an isomorphism
$$H^\unr_\ord(\bT)\otimes_{\co_\gp}S\iso H^\beta_g$$
(note that neither side depends on $\beta$, the right hand side by
Remark \ref{beta coords}).  This identification, together with the 
preceeding two lemmas, gives a commutative diagram
$$\xymatrix{ H^\unr_\ord(\bT)\otimes_\Lambda S\ar[r]\ar[d]& H^\beta_g
\otimes_{\Lambda^\beta_g}\Lambda^\beta_0\ar[d]\\
H^1_{\sel_\ord}(E,\bT)\otimes_\Lambda S\ar[r] &H^1_{\sel_S}(E,T_S)}$$
in which the upper horizontal arrow is an isomorphism, and the
two vertical arrows have finite kernel and cokernel, bounded as $i$
varies.  It follows that the bottom horizontal arrow has finite
kernel and cokernel, bounded as $i$ varies.

This concludes the proof of Proposition \ref{choosing specializations}.

\begin{Rem}\label{dual sequence}
In the construction of the sequence $\phi_i:\Lambda\map{}S$, 
if one replaces the ideal
$J$ by $J\cap J^\iota$ then the sequence $\phi_i$ still converges to 
$\gq$, and the sequence of dual specializations $\phi_i\circ\iota$
converges to $\gq^\iota$.
\end{Rem}


\subsection{The Main Conjecture}
\label{main conjecture section}


We continue to abbreviate  
$$X=\Hom_{\co_\gp}\big(H^1_{\sel_\ord}(E,\bW), \D_\gp\big),$$
and by $X_\tors$ the $\Lambda$-torsion submodule of $X$.
Let $H_\infty=H_\infty(1)\subset H^1_{\sel_\ord}(E,\bT)$ be the
$\Lambda$-module of Section \ref{iwasawa kolyvagin}.

\begin{Prop}\label{special bound}
Let $\phi:\Lambda\map{}S$ be a specialization such that the image of
$H_\infty$ under 
$H^1_{\sel_\ord}(E,\bT)\map{}H^1_{\sel_{S}}(E,T_{S})$
is nonzero.  Then $H^1_{\sel_S}(E,T_S)$ is a free rank one $S$-module, 
and there is an integer $d$ (independent of $\phi$) and 
a finite $S$-module $M_S$ such that
$$H^1_{\sel_S}(E,W_S)\iso \D_S\oplus M_S\oplus M_S$$
with $\length_S(M_S)\le \length_S\big(H^1_{\sel_S}(E,T_S)/
\phi(p^d H_\infty)\big)$.
\end{Prop}
\begin{proof} This is exactly as in \cite{howard}, 
and so we only give a sketch.
Fix a family $\{c_m\mid m\in\cm_1\}$ as in Proposition \ref{heegner module}.
As in Section \ref{heegner}, one may apply Kolyvagin's
derivative operators to obtain classes
$$\{\kappa_m'\in H^1(E,\bT/I_m\bT)\mid m\in\cm_1\}.$$
Lemma 2.3.4 of \cite{howard} asserts that 
\begin{equation}\label{iwasawa local verification}
\kappa_m'\in H^1_{\sel_\ord(m)}(E_v,\bT/I_m\bT),
\end{equation}
but the proof breaks down at primes of bad reduction which split completely
in $E_\infty$.
This is corrected as follows:
let $v$ be a prime of bad reduction which splits completely
in $E_\infty$ (so in particular $v$ does not divide $p$ or $m$).  
Choose $d$ large enough that $p^{d}$ annihilates
the finite group $H^2(E_v,T)$ (for all such choices of $v$).  
By the exactness of 
$$H^1(E_v,\bT)\map{}H^1(E_v,\bT/I_m\bT)\map{}H^2(E_v,\bT)$$
and the fact that $H^2(E_v,\bT)\iso H^2(E_v,T)\otimes\Lambda$,
we have that $p^{d}\kappa_m'$ lifts to $H^1(E_v,\bT)$.  By definition
of $\sel^\ord$, (\ref{iwasawa local verification})
now holds with $\kappa_m'$ replaced by $p^{d}\kappa_m'$.

By Lemma \ref{local control} the map $\bT\map{}T_S$ induces everywhere
locally a map $$H^1_{\sel_\ord}(E_v,\bT)\map{}H^1_{\sel_S}(E_v,T_S),$$
and therefore a map on global cohomology
$$H^1_{\sel_\ord(m)}(E,\bT/I_m\bT)\map{}H^1_{\sel_S(m)}(E,T_S/I_m T_S).$$
The images of the classes $p^d\kappa_m'$ may be modified, as in 
Theorem \ref{little money}, to form a Kolyvagin system for 
$(T_S,\sel_S,\cl_1)$, with $\kappa_1$ generating $\phi(p^d H_\infty)$.
The claim now follows from Theorem \ref{my thesis} and Lemma
\ref{special hypotheses}.
\end{proof}

\begin{Thm}\label{iwasawa structure}
There are torsion $\Lambda$-modules $M$ and $M_\gp$ such 
that $\gp$ does not divide $\mathrm{char}(M)$, $\mathrm{char}(M_\gp)
=\gp^k$ for some $k$, and $$X_\tors\piso M\oplus M\oplus M_\gp.$$  
Furthermore, $M$ satisfies the functional
equation $\mathrm{char}(M)=\mathrm{char}(M)^\iota.$
\end{Thm}
\begin{proof} 
Fix a height-one prime $\gq\not=\gp\Lambda$ with generator $Q$
and a pseudo-isomorphism
$$X\map{}\Lambda^r\oplus B\oplus C$$
with $B$ of the form $\bigoplus\Lambda/\gq^{e_k}$
and $C$ of the form $\bigoplus_\Lambda/f_k\Lambda$ with each $f_k\not\in\gq$.
Let $\phi_i:\Lambda\map{}S$ be the sequence of specializations converging
to $\gq$ constructed in Section \ref{choosing specs}.  In particular
$\phi_i$ satisfy the hypotheses of Lemma \ref{sequence lemma},
and so (by the proof of the lemma) the map
\begin{equation}\label{more structure}
X\otimes_\Lambda S\map{} S^{r_0}\oplus (B\otimes_\Lambda S)
\iso S^{r_0}\oplus \bigoplus S/\phi_i(Q)^{e_k}S\end{equation} 
has finite kernel and cokernel, bounded as $i$ varies.
On the other hand, Lemmas \ref{global control} and \ref{global control two}
give maps 
\begin{equation}\label{more structure two}X\otimes_\Lambda S\map{}
\Hom_S(H^1_{\sel_S}(E,W_S),\D_S)\end{equation}
 with finite kernel and cokernel,
bounded as $i$ varies.  The $S$-torsion submodule of this module
has the form $M_S\oplus M_S$ by Proposition \ref{structure} 
(and Lemma \ref{special hypotheses}).  
The maps (\ref{more structure}) and (\ref{more structure two}),
restricted to $S$-torsion, now give maps
\begin{eqnarray*}
(X\otimes_\Lambda S)_{S-\tors}&\map{}&\bigoplus S/\phi_i(Q)^{e_k}S\\
(X\otimes_\Lambda S)_{S-\tors}&\map{}&M_S\oplus M_S
\end{eqnarray*}
whose kernels and cokernels remain bounded as $i$ varies.
Letting $i\to\infty$, so that $\phi_i(Q)\to 0$, some elementary linear algebra
shows that each $e_k$ must occur as an exponent an even number of times.

For the functional equation,
choose a sequence of specializations $\phi_i:\Lambda\map{}S$ 
converging to $\gq$.  By Remark \ref{dual sequence} we may do this in such
a way that the sequence of dual specializations 
$\phi^*=\phi_i\circ\iota$ converges
to $\gq^\iota$.  Applying Proposition \ref{flach} and the definition of 
convergence (in particular hypothesis (c)), we have 
\begin{eqnarray*}\lefteqn{
\mathrm{ord}_\gq\big(\mathrm{char}(X_{\Lambda-\tors})\big)
\cdot \mathrm{length}_S\big(S/\phi_i(Q)S\big)}\hspace{2cm}\\
& &=\mathrm{ord}_{\gq^\iota}\big(\mathrm{char}(X_{\Lambda-\tors})\big)
\cdot \mathrm{length}_S\big(S/\phi_i^*(Q^\iota)S\big)
\end{eqnarray*}
up to $O(1)$ as $i$ varies.  Letting $i\to\infty$ gives the result.
 \end{proof}

\begin{Thm}\label{big money}
Assume that $h_k(1)\in A(E_k(1))$ has infinite order for some $k$, then 
\begin{enumerate}
\item the $\Lambda$-module $H^1_{\sel_\ord}(E,\bT)$ is torsion free 
of rank one,
\item $X\piso \Lambda\oplus X_\tors$,  
\item in the notation of Theorem
\ref{iwasawa structure}, $\mathrm{char}(M)$ divides
$\mathrm{char}\big(H^1_{\sel_\ord}(E,\bT)/H_\infty\big).$
\end{enumerate}
\end{Thm}
\begin{proof} 
By Proposition \ref{heegner module}, we are assuming that $H_\infty$
is a free rank one $\Lambda$-module.  Let $c$ be a generator.
By Proposition \ref{choosing specializations}, there is a specialization
$\phi:\Lambda\map{}S$ satisfying hypotheses {\bf Sp1,2,5}, and from 
Proposition \ref{special bound} we conclude that $H^1_{\sel_\ord}(E,\bT)$
and $X$ have $\Lambda$-rank one.  Furthermore, $H^1_{\sel_\ord}(E,\bT)$
has no $\Lambda$-torsion by Proposition \ref{lambda torsion}.

Now let $\gq\not=\gp\Lambda$ be a height-one prime of $\Lambda$,
and let $Q$ generate $\gq$.  Using
Proposition \ref{choosing specializations} we choose a sequence of
specializations $\phi_i:\Lambda\map{}S$ converging to $\gq$. 
Set $\sigma_i=\mathrm{length}_S(S/\phi_i(Q)S)$.
By Proposition \ref{special bound} the inequality
$$\mathrm{ord}_\gq\big(\mathrm{char}(X_{\Lambda-\tors})\big)\cdot\sigma_i\le
2\cdot \mathrm{ord}_\gq\big(\mathrm{char}
\big(H^1_{\sel_\ord}(E,\bT)/p^d H_\infty\big)\big)\cdot\sigma_i$$
holds up to $O(1)$ as $i$ varies.  As $i\to\infty$, $\sigma_i\to\infty$ 
and (since the factor of $p^d$ does not affect the order of 
the characteristic ideal at $\gq$) the result follows.
\end{proof}



\bibliographystyle{plain}

\end{document}